\documentclass[12pt, reqno]{amsart}
\usepackage{ifthen,amsfonts,
amsmath,amssymb, enumerate,
epsfig,color,graphicx,cite,bbm,graphics}
\usepackage{stmaryrd}
\usepackage{caption}
\usepackage{a4wide}
\usepackage{xfrac}
\usepackage{mathrsfs,bm,amsthm,mathtools,yfonts}
\usepackage{xcolor}
\mathtoolsset{showonlyrefs,showmanualtags}

\makeatletter

\begingroup

\newtheorem{theorem}{Theorem}[section]
\newtheorem{lemma}[theorem]{Lemma}
\newtheorem{propos}[theorem]{Proposition}
\newtheorem{corol}[theorem]{Corollary}
\endgroup

\newtheorem{definition}[theorem]{Definition}

\newtheorem{remark}[theorem]{Remark}

\def\RR{{\mathbb{R}}}

\def\TT{{\mathbb  T}}

\def\NN{{\mathbb{N}}}
\newcommand{\eps}{{\varepsilon}}

\newcommand{\Leb}[1]{{\mathscr L}^{#1}} 

\newcommand{\fR}{{\mathcal R}}
\newcommand{\supp}{{\rm supp}}
\newcommand{\no}{10}
\newcommand{\nop}{12}

\renewcommand{\div}{{\text {div}}\,}

\makeindex
\begin{document}

\title[]
{Positive solutions of transport equations and classical nonuniqueness of characteristic curves}
\author{Elia Bru\'e}
\address[Elia Bru\'e]{
\newline \indent Scuola Normale Superiore
\newline \indent Piazza dei Cavalieri 7, 56126 Pisa, Italy}
\email{elia.brue@sns.it}
\author{Maria Colombo}
\address[Maria Colombo]{
\newline \indent EPFL B, Station 8
\newline \indent CH-1015 Lausanne, CH}
\email{maria.colombo@epfl.ch}
\author{Camillo De Lellis}
\address[Camillo De Lellis]{
\newline \indent School of Mathematics, Institute for Advanced Study and Universit\"at Z\"urich 
\newline \indent 1 Einstein Dr., Princeton NJ 05840, USA}
\email{camillo.delellis@math.ias.edu}

\begin{abstract}
The seminal work of DiPerna and Lions \cite{DL} guarantees the existence and uniqueness of regular Lagrangian flows for Sobolev vector fields. The latter is a suitable selection of trajectories of the related ODE satisfying additional compressibility/semigroup properties. A long-standing open question is whether the uniqueness of the regular Lagrangian flow is a corollary of the uniqueness of the trajectory of the ODE for a.e. initial datum. Using Ambrosio's superposition principle we relate the latter to the uniqueness of positive solutions of the continuity equation and we then provide a negative answer using tools introduced by Modena and Sz\'ekelyhidi in the recent groundbreaking work \cite{MoSz2019AnnPDE}. On the opposite side, we introduce a new class of asymmetric Lusin-Lipschitz inequalities and use them to prove the uniqueness of positive solutions of the continuity equation in an integrability range which goes beyond the DiPerna-Lions theory.
\end{abstract}
\maketitle


\section{Introduction}
In this paper we study positive solutions of the continuity equation
\begin{equation}
\label{eqn:CE}\partial_t \rho + \div (u \rho) = 0 
\end{equation}
and the related system of ordinary differential equations $\dot\gamma(t) = u(t, \gamma(t))$.
To avoid technicalities we restrict our attention to {\em periodic} vector fields, i.e. $u: I \times \TT^d \to \mathbb R^d$, where 
$\TT^d$ is the $d$-dimensional torus and $I\subset \RR$. In the sequel we omit the superscript $d$ and simply write $\TT$ and use the notation $\Leb{d}$ for the Lebesgue measure on the whole space $\mathbb R^d$ and on $\TT$. 

\begin{definition}\label{defn:int-curve}
Let  $u: (0,T) \times \TT \to \RR^d$ be a Borel map. We say that $\gamma \in AC([0,T]; \TT)$ is an integral curve of $u$ starting at $x$ if $\gamma(0)=x$ and $\gamma'(t) = u(t, \gamma(t))$ for a.e. $t\in [0,T]$.
\end{definition}

Note that in Definition \ref{defn:int-curve} it matters how $u$ is defined at {\em every point}: different pointwise representatives for $u$ might have different integral curves starting at the same $x$. When $u$ is smooth (Lipschitz) the trajectories are unique and, after ``bundling them'' into a flow map $X: (0,T)\times \TT\to \TT$, solutions of \eqref{eqn:CE} can be recovered via Liouville's classical theorem. This fact can be elegantly encoded using measure theory in the formula $(X (t, \cdot))_\sharp (\rho (0, \cdot) \Leb{d}) = \rho (t, \cdot) \Leb{d}$. For less regular vector fields it is customary, after the seminal paper \cite{DL,Ambr-inv2004}, to introduce the notion of {\em regular Lagrangian flows}. The latter consists, following one of its equivalent formulations given in \cite{Ambr-inv2004}, of a measurable selection $X$ of integral curves of the ODE for which $X(t, \cdot)_\# \Leb{d} \leq C \Leb{d}$.

\begin{definition}\label{def:regflow}
Let  $u: (0,T) \times \TT \to \RR^d$ be Borel. $X: [0,T]\times \TT\to \TT$ is a {\em 
regular Lagrangian flow} of $u$ if
\begin{itemize}
\item[(i)] for $\Leb{d}$-a.e. $x\in \TT$, $X(\cdot,x)\in AC([0,T];\TT)$ is an integral curve of $u$ starting at $x$;
\item[(ii)] there exists a constant $C=C(X)$ satisfying $X(t,\cdot)_\#\Leb{d}\leq C\Leb{d}$.
\end{itemize}
\end{definition}

The pointwise definition of $u$ matters in Definition \ref{def:regflow} as well. However, it is an outcome of the DiPerna-Lions theory that, under suitable Sobolev regularity assumptions on $u$, regular Lagrangian flows exist, satisfy a semigroup property, are unique, stable under approximations, and independent of the pointwise representative chosen for $u$. 

Such uniqueness and stability result is sometimes inappropriately regarded as ``almost everywhere uniqueness of integral curves'', even though it is well known among the experts that the DiPerna-Lions theory does not imply the statement ``for a.e. $x$ there is a unique integral curve of $u$ starting at $x$''. In fact whether such ``classical'' uniqueness theorem holds for Sobolev vector fields is a long-standing open question, see \cite[p. 546]{DL}, \cite[p.231]{Ambr-inv2004}, \cite[Section 2.3]{Alb2012Lincei}, 
\cite[Open problems, section 4]{Luigi-CIME}. This question has had a positive answer for specific vector fields, such as suitable weak solutions of the Navier-Stokes system \cite{RoSa2009nonlin,RoSa2009}, based on estimates of the dimension of the singular set originally due to \cite{CaffarelliKohnNirenberg82}. Recently, in \cite{CC18} the authors use a suggestion of Jabin to prove almost everywhere uniqueness of the trajectories when $u\in C ([0,T], W^{1,r} (\TT, \RR^d))$ for some $r>d$.  One aim of this paper is to show that in general, under the assumptions of the DiPerna-Lions theory, the answer is negative. 

\begin{theorem}\label{thm:main-flow}
For every ${{d\geq 2}}$, $r< d$, $s<\infty$ and every $T>0$ there is a divergence-free vector field $u \in C ([0,T], W^{1,r} (\TT, \RR^d)\cap L^s)$ such that the following holds for every Borel map $v$ with $u=v$ $\Leb{d+1}$-a.e.:
\begin{itemize}
\item[(NU)] There is a measurable $A \subseteq \TT$ with positive Lebesgue measure such that for every $x \in A$ there are at least two integral curves of $v$ starting at $x$. 
\end{itemize}
\end{theorem}

Given \cite[Theorem 5.2]{CC18}, the above statement covers the optimal range, except for the endpoint $r=d$. In fact an improvement of the argument in \cite[Theorem 5.2]{CC18} allows to prove almost everywhere uniqueness of trajectories for a function space which shares the same scaling properties of $W^{1,d}$, namely when $Du$ belongs to the Lorentz space $L^{d,1}$, see Corollary \ref{c:Lorentz} below. We also expect Theorem \ref{thm:main-flow} to hold for some Sobolev vector fields that are continuous in space-time \cite{future}, thereby answering the question of \cite[Section 2.3]{Alb2012Lincei} and providing an example under the classical assumptions of Peano's theorem.

The theorem above is a consequence of Ambrosio's superposition principle (see \cite[Theorem 3.2]{Luigi-CIME}) and of the following nonuniqueness result at the PDE level, which in turn will be proved using ``convex integration type'' techniques borrowed from a groundbreaking work of Modena and Sz\'ekelyhidi \cite{MoSz2019AnnPDE,MoSz2019CalcVar}, improved later by Modena and Sattig \cite{MoSa2019}(we refer to \cite{DLSZ09,DLSZ13,DSZ17,Isett2018Annals,BDLSV17,BV} and the references therein for the birth of this and related lines of research). 

\begin{theorem}\label{thm:main-CE}
Let 
${{d\geq 2}}$, $p \in (1,\infty), r \in [1,\infty]$ be such that 
\begin{equation}
\label{hp:exp}
\frac 1p + \frac 1r > 1+ \frac 1{d}
\end{equation}
and denote by $p'$ the dual exponent of $p$, i.e. $\frac{1}{p}+\frac{1}{p'}=1$. 
Then for every $T>0$ there exists a divergence-free vector field $u \in C ([0,T], W^{1,r} (\TT, \RR^d)\cap L^{p'})$ and a {\em nonconstant} $\rho \in C ([0,T], L^p (\TT))$ such that \eqref{eqn:CE} holds with initial data $\rho(0, \cdot) =1$ and for which $\rho \geq c_0$ for some positive constant $c_0$.
\end{theorem}
Compared to the results in \cite{MoSz2019AnnPDE} and \cite{MoSa2019} the addition (crucial for our application) is the positivity of the solution $\rho$.
While it is relatively simple to modify the approach of Modena and Sz\'ekelyhidi in \cite{MoSz2019AnnPDE} in order to achieve Theorem \ref{thm:main-CE} when $\frac{1}{p}+\frac{1}{r} > 1 + \frac{1}{d-1}$, we have not been able to do the same with the one in \cite{MoSa2019} to cover the range $1+\frac{1}{d-1} \geq \frac{1}{p}+\frac{1}{r} > 1 + \frac{1}{d}$. Our proof is therefore relatively different from the one of \cite{MoSa2019} and in fact less complicated and shorter. At the technical level we introduce suitable space-time flows which compared to the basic building blocks of \cite{MoSa2019} are more similar to Mikado flows: in a nutshell our flows are a perturbation of point masses traveling on a space-time line.
This approach makes a part of our argument more similar to \cite{MoSz2019AnnPDE}, but it has the technical drawback that we need to introduce a suitable partition of unity to discretize the time velocities of the moving particles (a similar idea was used first in \cite{DLSZ13}). One subtle part of our proof is a combinatorial argument to ensure that the supports of the flows are disjoint in $2$ space dimensions. Since in $3$ space dimensions and higher the latter can be completely omitted and the proof is simpler we have decided to first present the full arguments for Theorem \ref{thm:main-CE} when $d\geq 3$ and then show in Section \ref{s:2d} which modifications are necessary in the case $d=2$. 

\medskip

Our interest in Theorem \ref{thm:main-CE} was triggered by the gap between the DiPerna-Lions theory, which guarantees uniqueness for $\frac{1}{p}+\frac{1}{r} \leq 1$, and the nonuniqueness results of \cite{MoSz2019AnnPDE,MoSz2019CalcVar,MoSa2019}. In particular we are able to show that in some intermediate range of exponents (strictly containing the DiPerna-Lions range, but not reaching the full complement of the Modena-Sattig-Sz\'ekelyhidi range) {\em positive} solutions are in fact unique.

\begin{theorem}\label{thm:wellposedness}
	Let $d\ge 2$, $p\in [1,+\infty]$ and $r\in [1,+\infty]$ be such that
	\begin{equation}\label{eq:wellposednessrange}
	\frac{1}{p}+\frac{1}{r}<1+\frac{1}{d-1}\frac{r-1}{r}.
	\end{equation}
	Let {$u\in L^1([0,T],W^{1,r}(\TT,\RR^d))$} be a vector field satisfying $\div\, u\in L^{\infty}$. Then, for $p>1$, \eqref{eqn:CE} admits a unique solution among all nonnegative, weakly continuous in time densities $\rho\in L^{\infty}([0,T],L^p(\TT))$ with $\rho(0,\cdot)=\rho_0$. When $p=1$ ({i.e. $r>d$}) uniqueness holds in the class of nonnegative weakly-star continuous densities $\rho\in L^{\infty}([0,T],\mathscr{M}(\TT
	))$ with $\rho(0,\cdot)=\rho_0\Leb d$. In particular, any such $\rho$ is Lagrangian, i.e.
	\begin{equation*}
	\rho(t,\cdot )\Leb d=X(t,\cdot)_{\#}(\rho_0 \Leb d)
	\quad \text{for every $t\in [0,T]$},
	\end{equation*}
	where $X$ denotes the unique regular Lagrangian flow of Definition \ref{def:regflow}.
\end{theorem}
{\begin{remark} Observe that, under the above assumptions $u\in L^1 ([0,T], L^{p'})$. Indeed, if $r>d$ Morrey's embedding guarantees $u\in L^1 ([0,1], L^q)$ for every $q\in [1, \infty]$ and if
$r\leq d$ Sobolev's embedding guarantees $u\in L^1 ([0,T], L^{q})$ for every $q< \frac{rd}{d-r}$ while \eqref{eq:wellposednessrange} is equivalent to $p'< \frac{r(d-1)}{d-r}$.
\end{remark}}
Theorem \ref{thm:main-CE} extends \cite[Corollary 5.4]{CC18}, in which the case $r>d$ has been settled as a consequence of the $\Leb d$-a.e. uniqueness result for trajectories mentioned above. 
The proofs of the latter and of Theorem \ref{thm:wellposedness} employ all some suitable Lusin-Lipschitz type estimates for $u$, an idea pioneered in \cite{ALM2005} and \cite{CD18} and which has proved quite fruitful in different contexts (see for instance \cite{BoCr2013,bru2018constancy,CoCrSp2016,DLGw22016,BrNg2018}). As it is well known, for sufficiently regular domains $\Omega\subset \RR^d$ and when $p\in (1, \infty]$, a Borel map $u$ belongs to $W^{1,p} (\Omega, \RR)$ if and only if there is a function $g\in L^p (\Omega)$ such that
\begin{equation}\label{e:simmetrica}
|u(x)-u(y)|\leq (g(x)+g(y))|x-y| \qquad \mbox{for a.e. $x,y$.}
\end{equation}
In fact $g$ can be taken to be the classical Hardy-Littlewood maximal function of $|Du|$.
It seems less known (but anyway classical) that for $p>d$ the symmetry in \eqref{e:simmetrica} can be broken to show
\begin{equation}\label{e:tot_asimmetrica}
|u(x) - u(y)|\leq g(x) |x-y|\, .
\end{equation}
Theorem \ref{thm:wellposedness} is based on the idea that an appropriate symmetry-breaking is still possible for smaller exponents $p$. More precisely we have the following proposition, which has its own independent interest. 

\begin{propos}\label{prop:LusinLipschitzasimmetrica}
	Let $1<r\le d$ be fixed. For any $u\in W^{1,r}(\TT)$ and any $\alpha\in (0,\frac{r}{d})$ there exist a negligible set $N\subset \TT$ and a nonnegative function $g\in L^r (\TT)$ satisfying the inequalities
	\begin{equation*}
	\| g\|_{L^r}\le C(\alpha,r,d)\| D u \|_{L^r},
	\end{equation*} 
	\begin{equation}\label{eq:LusinLipparzialmenteAsimmetrica}
	|u(x)-u(y)|\le |x-y|\left( g(x)+g(x)^{\alpha}g(y)^{1-\alpha}\right)
	\quad
	\text{for any}\ x,y\in \RR^d\setminus N.
	\end{equation}
	Moreover, we can assume $N=\emptyset$ provided we choose an appropriate representative of $u\in W^{1,r}(\TT)$ and there is a continuous selection $W^{1,r} \ni u\mapsto g\in L^r$. 
\end{propos}

A simple corollary of the latter statement is an inequality of the form $|u(x)-u(y)|\leq (a(x)+b(y))|x-y|$ where one function, say $b$, can be taken more integrable at the prize of giving up some integrability for the other.
Theorem \ref{thm:wellposedness} follows from the extreme case where the integrability of $b$ is maximized at the expense of reducing the integrability of $a$ to the bare minimum, namely $L^1$, cf. Corollary \ref{cor:LusinLipsLinearversion}. We moreover show that in this case the range of exponents for $b$ obtained in the latter is in fact optimal. 

{Clearly, it is tempting to advance the conjecture that, for positive solutions of the continuity equations, 
well-posedness holds in the range $1+\frac{1}{d} > \frac{1}{p}+\frac{1}{r}$,
namely the complement of the the closure of the range of Theorem \ref{thm:main-CE}. An even more daring conjecture is that the latter statement holds for any solution. However nothing is known without assuming a one-sided bound or, as is the case of \cite{CC18}, some technical property of trajectories of the ODEs. 
}

\section{Iteration and continuity-Reynolds system} 

As in \cite{MoSz2019AnnPDE} we consider the following system of equations in $[0,T] \times \TT$
\begin{equation}\label{eqn:CE-R}
\begin{cases}
\partial_t \rho_{q} + \div(  \rho_{q} u_{q})=-\div R_{q}
\\ \\
\div u_q=0.
\end{cases}
\end{equation}
We then fix three parameters $a_0$, $b>0$ and $\beta>0$, to be chosen later only in terms of $d$, $p$, $r$, and for any choice of $a>a_0$ we define
\[
\lambda_0 =a, \quad \lambda_{q+1} = \lambda_q^b \quad \mbox{and}\quad
\delta_{q} = \lambda_{q}^{-2\beta}\, .
\]
The following proposition builds a converging sequence of functions with the inductive estimates
\begin{equation}
\label{eqn:ie-1}
\max_t \|R_q (t, \cdot)\|_{L^1} \leq \delta_{q+1}
\end{equation}
\begin{equation}
\label{eqn:ie-2}
\max_t \left(\| \rho_q (t, \cdot)\|_{C^1} +\|\partial_t \rho_q (t, \cdot)\|_{C^0} + \|u_q (t, \cdot)\|_{W^{1,p'}} + \| u_q (t, \cdot)\|_{W^{2,r}} + \|\partial_t u_q (t, \cdot)\|_{L^1}\right) 
\leq \lambda_{q}^\alpha\, ,
\end{equation}
where $\alpha$ is yet another positive parameter which will be specified later. 

\begin{propos}\label{prop:inductive}
	There exist $\alpha,b, a_0, M>5$, $0<\beta<(2b)^{-1}$ such that the following holds. For every $a\geq a_0$, if
	$(\rho_q, u_q, R_q)$ solves \eqref{eqn:CE-R} and enjoys the estimates \eqref{eqn:ie-1}, \eqref{eqn:ie-2}, then there exist $(\rho_{q+1}, u_{q+1}, R_{q+1})$ which solves \eqref{eqn:CE-R}, enjoys the estimates \eqref{eqn:ie-1}, \eqref{eqn:ie-2} with $q$ replaced by $q+1$ and also the following properties:
	\begin{itemize}
		\item[(a)] $\max_t [ \|( \rho_{q+1}- \rho_q) (t, \cdot)\|_{L^p}^p +\|(u_{q+1}- u_q) (t, \cdot)\|_{W^{1,r}}^r +\|(u_{q+1}- u_q) (t, \cdot)\|_{L^{p'}}^{p'} ] \leq M \delta_{q+1}$
		\item[(b)] $\inf (\rho_{q+1} - \rho_q) \geq - \delta_{q+1}^{1/p}$
		\item[(c)] if for some $t_0>0$ we have that $\rho_q(t, \cdot) = 1$, $R_q(t, \cdot)=0$ and $u_q(t, \cdot)=0$ for every $t\in [0,t_0]$, then $\rho_{q+1}(t, \cdot) = 1$, $R_{q+1}(t, \cdot)$ and $u_{q+1}(t, \cdot)=0$ for every $t\in [0,t_0- \lambda_q^{-1-\alpha}]$.
	\end{itemize}
\end{propos}

Compared to \cite{MoSz2019AnnPDE} we are using a slightly different notation and a more specific choice of the parameters. None of that is however substantial: the really relevant differences are in estimate (b) and in the range of exponents, which is the same as the one in \cite{MoSa2019}. In the same range of exponents of \cite{MoSz2019AnnPDE} the positivity could be achieved by a slight tweak in the approach of \cite{MoSz2019AnnPDE}. However we have not been able to find a similar modification of the arguments of \cite{MoSa2019}. For this reason our proof of Proposition \ref{prop:inductive} differs from both that of \cite{MoSz2019AnnPDE} and that of \cite{MoSa2019}. However we still make use of some crucial discoveries in \cite{MoSz2019AnnPDE} and we will refer to that paper for the proofs of some relevant lemmas.
From now on, in order to simplify our notation, for any function space $X$ and any map $f$ which depends on $t$ and $x$, we will write $\|f\|_X$ meaning $\max_t \|f (t, \cdot)\|_X$.

\section{Preliminary lemmas}

\subsection{Geometric lemma} We start with an elementary geometric fact, namely that every vector in $\mathbb R^d$ can be written as a ``positive'' linear combination of elements in a suitably chosen finite subset $\Lambda$ of $\mathbb{Q}^d\cap \partial B_1$. This is reminiscent of the geometric lemma in \cite{DLSZ13}. In both \cite{MoSz2019AnnPDE} and \cite{MoSa2019} the positivity of the coefficients is not needed and hence the authors can choose $\Lambda$ as the standard basis of $\mathbb R^d$. 

\begin{lemma}\label{lemma:geom}
	There exists a finite set $\{ \xi\}_{\xi \in \Lambda} \subseteq \partial B_1\cap \mathbb{Q}^d$ and smooth nonnegative coefficients $a_\xi(R)$ such that for every $R \in \partial B_1$ 
	\[
	R= \sum_{\xi \in \Lambda} a_\xi(R) \xi\, .
	\]
\end{lemma}
\begin{proof} For each vector $v$ consider a collection $\Lambda (v)=\{\xi_1 (v), \ldots , \xi_d (v)\}\subset \partial B_1$ of linearly independent unit vectors in $\mathbb Q^d$ with the property that the $d$-dimensional open symplex $\Sigma (v)$ with vertices $0, 2\xi_1 (v), \ldots , 2\xi_d (v)$ contains $v$. Since $\{\Sigma (v): v\in \partial B_1\}$ is an open cover of $\partial B_1$, we consider a finite subcover and the corresponding collections $\Lambda_1 = \Lambda (v_1), \ldots , \Lambda (v_N)$, each one consisting of the $d$ unit vectors $\{ \xi_{j,1}, ... \xi_{j,d}\}$. We set 
	\[
	\Lambda = \bigcup_{j=1}^N \Lambda_j = \left\{ \xi_j,i: 1\leq j \leq N, 1 \leq i \leq d\right\}\,.
	\]
	For each fixed $j$ each vector $R\in \mathbb R^d$ can be written in a unique way as linear combination of the vectors $\xi_{j,1}, \ldots \xi_{j,d}$. If we denote by $b_{j,i} (R)$ the corresponding coefficients (which obviously depend linearly on $R$), then the latter are all strictly positive if $R$ belongs to $\Sigma (v_j)$.
	We consider a partition of unity $\chi_j$ on the unit sphere $\partial B_1$ associated to the cover $\{\Sigma (v_j)\}$ and for every $\xi_{j,i} = \xi \in \Lambda$ we set
	\[
	a_\xi (R) := \chi_j (R) b_{j,i} (R).
	\]
	The coefficients $a_\xi$ are then smooth nonnegative functions of $R$.
\end{proof}

\begin{remark}
With Lemma~\ref{lemma:geom} at hand, it is easy to generate a finite number of disjoint families 
$\Lambda^{(1)}, ..., \Lambda^{(k)}$ where each one enjoys the property of Lemma~\ref{lemma:geom}: it is enough to take suitable rational rotations of one fixed set $\Lambda$.
\end{remark}

\subsection{Antidivergences}
We recall that the operator $\nabla \Delta^{-1}$ is an antidivergence when applied to smooth vector fields of $0$ mean. As shown in \cite[Lemma 2.3]{MoSz2019AnnPDE} and \cite[Lemma 3.5]{MoSa2019}, however, the following lemma introduces an improved antidivergence operator, for functions with a particular structure.

\begin{lemma}\label{lemma23}(Cp. with \cite[Lemma 3.5]{MoSa2019})
	Let $\lambda \in \NN$ and $f, g : \TT \to \RR$ be smooth functions, and $g_\lambda= g(\lambda x)$. Assume that $\int g = 0$. Then if we set $\fR (f g_\lambda) = f \nabla \Delta^{-1} g_\lambda -\nabla \Delta^{-1} (\nabla f \cdot \nabla \Delta^{-1}g_\lambda+\int fg_\lambda)$, we have that $\div  \fR (f g_\lambda) = f g_\lambda-\int fg_\lambda$ and for some $C:=C({k,p})$
	\begin{equation}
	\label{ts:antidiv}
	\|D^k \fR (f g_\lambda)\|_{L^p} \leq C \lambda^{k-1} \|f\|_{C^{k+1}} \| g\|_{W^{k,p}} \qquad \mbox{for every } k\in \NN, p\in [1,\infty].
	\end{equation}
\end{lemma}
\begin{proof}
	It is enough to combine \cite[Lemma 3.5]{MoSa2019} and the remark in \cite[page 12]{MoSa2019}.
\end{proof}

\subsection{Slow and fast variables}
Finally we recall the following improved H\"older inequality, stated as in \cite[Lemma 2.6]{MoSz2019AnnPDE} (see also \cite[Lemma 3.7]{BV}).
If $\lambda \in \NN$ and $f,g:\TT \to \RR$ are smooth functions, then we have 
\begin{equation}
\label{eqn:impr-holder}
\| f(x) g(\lambda x) \|_{L^p} \leq \| f \|_{L^p} \| g \|_{L^p} + \frac{C(p)\sqrt d \|f \|_{C^1} \|g\|_{L^p}}{\lambda^{1/p}}
\end{equation}
and 
\begin{equation}
\label{eqn:l26}
\Big| \int f(x) g(\lambda x) \, dx \Big| \leq\Big| \int f(x) \Big(g(\lambda x) - \int g \Big) \, dx \Big| + \Big| \int f \Big |\cdot \Big| \int g \Big| \leq \frac{\sqrt d \|f \|_{C^1} \|g\|_{L^1}}{\lambda} + \Big| \int f \Big| \cdot \Big | \int g \Big| .
\end{equation}

\section{Building blocks}\label{sec:build}
Let $0<\rho<\frac 14$ be a constant
. We consider $\varphi\in C^\infty_c (B_\rho)$ and $\psi\in C^\infty_c (B_{2\rho})$ which satisfy
	$$\int \varphi =1, \qquad \varphi \geq 0, \qquad
	\psi \equiv 1 \mbox{ on }B_\rho.$$
Given $\mu\ll1$ we define the $1$-periodic functions
\begin{align*}
\bar \varphi_\mu (x) &:= \sum_{k\in \mathbb Z^d} \mu^{d/p} \varphi (\mu (x+k))\\
\bar \psi_\mu (x) &:= \sum_{k\in \mathbb Z^d} \mu^{d/p'} \psi (\mu (x+k))\, .
\end{align*}
Let $\omega: \mathbb R^d \to \mathbb R$ be a  smooth $1$-periodic function such that $\omega(x)  = x\cdot \xi'$ on $B_{2\rho}(0)$.

Given $\Lambda$ as in Lemma~\ref{lemma:geom}, for any $\xi\in \Lambda$ we chose $\xi'\in \partial B_1$ such that $\xi\cdot \xi'=0$ and we define
\[
\Omega_{\xi}^\mu(x):= \mu^{-1} \omega(\mu\, x) ( \xi\otimes \xi' - \xi'\otimes \xi ).
\]
Notice that $\div \Omega^\mu_{\xi}$ is divergence free since $\Omega_{\xi}^\mu$ is skew-symmetric and $\div \Omega_{\xi}^\mu=\xi$ on $\supp (\bar \psi_\mu)$ and $\supp (\bar \varphi_{\mu})$.

For $\sigma>0$ we set 
\begin{align*}
\tilde W_{\xi, \mu, \sigma} (t,x) &:= \sigma^{1/p'}  \div \left[ (\Omega_{\xi}^\mu \bar \psi_\mu) (x- \mu^{d/p'} \sigma^{1/p'} t \xi )\right]\\
\tilde \Theta_{\xi, \mu, \sigma} (t,x) &:= \sigma^{1/p} \bar \varphi_\mu (x - \mu^{d/p'} \sigma^{1/p'} t \xi )\, .
\end{align*}
Notice that $\tilde W_{\xi, \mu, \sigma} $ is divergence free since it is also the divergence of the skew-symmetric matrix $\Omega_{\xi}^\mu \bar \psi_\mu$. By construction we have 
\[
\tilde W_{\xi, \mu, \sigma} (t,x)=\sigma^{1/p'}\left[ \bar \psi_{\mu} \xi + \Omega_{\xi}^{\mu} \cdot \nabla \bar \psi_{\mu}  \right](x-\mu^{d/p'} \sigma^{1/p'} t \xi ),
\]
hence the following properties are easily verified.

\begin{lemma}\label{lemma:build} We have
	\begin{equation}\label{eqn:itsolves}
	\partial_t \tilde\Theta_{\xi, \mu, \sigma}+\div(\tilde W_{\xi, \mu, \sigma} \tilde\Theta_{\xi, \mu, \sigma})=0,
	\end{equation}
 $$\div \tilde W_{\xi, \mu, \sigma}=0,$$
 \begin{equation}
 \label{eqn:avW}
 \int \tilde W_{\xi, \mu, \sigma}=0,
 \end{equation}
$\tilde W_{\xi, \mu, \sigma} \tilde\Theta_{\xi, \mu, \sigma}(t,x) = \sigma \mu^{d/p'}\bar{\varphi}_{\mu}(x-\mu^{d/p'}\sigma^{1/p'}t\xi)\xi$, in particular
	\begin{equation}\label{eqn:rightaverage}
	\int \tilde W_{\xi, \mu, \sigma} \tilde\Theta_{\xi, \mu, \sigma}=\sigma \xi \int \varphi= \sigma \xi.
	\end{equation}
For any $k\in \NN$ and any $s\in [1,\infty]$ one has
\begin{equation}\label{eqn:Thetanorms}
\| D^k \tilde\Theta_{\xi, \mu, \sigma} \|_{L^s}\le C(d,k,s) \sigma^{1/p} \mu^{k + d ( 1/p - 1/s )},
\qquad
\|\partial_t^k \tilde\Theta_{\xi, \mu, \sigma}\|_{L^s} \le C(d,k,s) \sigma^{1 + \frac{k - 1}{ p' }} \mu^{k + d ( \frac{ k - 1 }{ p' } + 1 - \frac{1}{s} ) }
\end{equation}

\begin{equation}\label{eqn:Wnorms}
\| D^k \tilde W_{\xi,\mu,\sigma} \|_{L^s} \le C(d,k,s) \sigma^{1/p'} \mu^{ k+ d(1/p'-1/s)},
\qquad
\| \partial_t^k  \tilde W_{\xi,\mu,\sigma} \|_{L^s}\le C(d,k,s) \sigma^{\frac{k + 1}{ p' }} \mu^{ k + d(\frac{k+1}{p'}-\frac{1}{s} ) }.
\end{equation}
Finally, $\supp \Theta_{\xi, \mu, \sigma} \cup \supp W_{\xi, \mu, \sigma} \subseteq \{ (x,t): x- \mu^{d/p'} \sigma^{1/p'} t \xi\in B_{2\rho \mu^{-1}} + \mathbb Z^d  \} $ and the support in space is contained in a periodized cylinder
\begin{equation}
\label{eqn:support}
\big\{x: \tilde W_{\xi, \mu, \sigma}(x,t) \neq 0 \mbox{ or }\tilde \Theta_{\xi, \mu, \sigma}(x,t) \neq 0 \mbox{ for some }t\geq 0\big \} \subseteq 
 B_{2\rho \mu^{-1}} + \mathbb R \xi + \mathbb Z^d
\end{equation}
\end{lemma}

In our construction $\xi$ will take values in a finite set of $\xi$'s, which will be fixed throughout the iteration, i.e. it is independent of the step $q$ in Proposition \ref{prop:inductive}. The parameter $\sigma$ will also vary in a finite set, but the cardinality of the latter will depend (and in fact diverge to infinity) on the iteration step $q$. In dimension $d\geq 3$ we 
consider suitable translations of $\tilde W_{\xi, \mu, \sigma}$ and $\tilde\Theta_{\xi, \mu, \sigma}$ which guarantee that, as $\xi$ varies in these fixed set of directions, pairs of $(\tilde{W}, \tilde{\Theta})$ with distinct $\xi$'s have disjoint supports. The precise statement is given in the following lemma.

\begin{lemma}\label{lemma:disjointsupports}
Let $d \geq 3$ and $\Lambda \subseteq \mathbb S^{d-1} \cap \mathbb Q$ be a finite number of vectors. Then there exists $\mu_0:= \mu_0 (d , \Lambda) >0$ and a family of vectors $\{v_\xi\}_{\xi \in \Lambda} \subseteq \RR^d$ such that the periodized cylinders 
$v_\xi+ B_{2\rho \mu^{-1}} + \mathbb R \xi + \mathbb Z^d
$ are disjoint as $\xi$ varies in $\Lambda$, provided $\mu \geq \mu_0$. 
\end{lemma}
\begin{proof}
	Set $\ell(v,\xi):=v+\RR \xi +\mathbb Z^d$. It is enough to find $\{v_{\xi}\}_{\xi\in \Lambda}\subset \RR^d$ such that $\ell(v_{\xi},\xi)\cap \ell(v_{\xi'},\xi')=\emptyset$ whenever $\xi\neq \xi'$. This claim follows from a simple induction argument along with the observation:
	\begin{equation}\label{z3}
		\Leb d (\RR^d \setminus \{ v'\in \RR^d : \, \ell(v,\xi)\cap \ell(v',\xi')=\emptyset \})=0
		\quad \text{for any $v\in \RR^d$ and any $\xi, \xi'\in \mathbb S^{d-1}\cap \mathbb{Q}^d$, $\xi\neq \xi'$.}
	\end{equation}
	To verify \eqref{z3} we notice that $\ell(v,\xi)\cap \ell(v',\xi')=\emptyset$ if and only if for every $s,t\in \RR$, $k\in \mathbb{Z}^d$ the inequality $v'-v\neq t\xi-s\xi'+k$ holds. In particular any $v'=v+\alpha \xi+\beta \xi'+ \gamma \xi''$ with $\alpha, \beta\in \RR$, $\gamma\in \RR\setminus \mathbb{Q}$ and $\xi''\in \mathbb{Q}^d\setminus \{ 0 \}$ orthogonal to $\xi$ and $\xi'$, has this property. Indeed, if we assume by contradiction the existence of $s,t,k,\alpha,\beta,\gamma,\xi''$ as above such that $\alpha \xi + \beta \xi'+\gamma \xi''= t\xi-s\xi'+k$ we get  $\gamma |\xi''|^2=k\cdot \xi''\in \mathbb{Q}$ that contradicts $\gamma\in \RR\setminus \mathbb{Q}$ and $\xi''\in \mathbb{Q}^d\setminus \{ 0 \}$.
\end{proof}
By the previous lemma and by \eqref{eqn:support} we notice that, if we consider the translations of 
$$
W_{\xi, \mu, \sigma}(t, x) =\tilde W_{\xi, \mu, \sigma}(t, x-v_\xi), \qquad
\tilde\Theta_{\xi, \mu, \sigma} (t,x) = \tilde\Theta_{\xi, \mu, \sigma}(t, x-v_\xi),$$
for $\xi$ in a suitable finite set of directions, these functions satisfy the same properties as in Lemma~\ref{lemma:build} (with the exception of the description of the support, which is now translated) and moreover they have disjoint support for $\mu$ sufficiently large and for every $\sigma$. Notice finally that in fact both $\mu$ and $\sigma$ could vary for different $\xi$ and the supports would still remain disjoint, as long as $\mu (\xi)$ is larger than $\mu_0$ for every $\xi$. 

The latter approach is clearly not feasible in dimension $d= 2$. In that case we will need to take advantage of the discreteness in the parameter $\sigma$ as well. As already mentioned this is more delicate, since the set of values taken by $\sigma$ depends on the step $q$. At each step $q$ we need to choose rather carefully the set of parameters $\sigma$ which enter the construction: for distinct values of $\sigma$ we need to ensure that their ratio is not too close to $1$, compared to the size of $\mu^{-1}$. The relevant statements depend thus on how the building blocks enter in the definition of the maps $(u_{q+1},  \rho_{q+1}, R_{q+1})$. For this reason we detail next the definition of the maps when $d\geq 3$ and show first how to prove Proposition \ref{prop:inductive} in that case. We then give a detailed description on how to modify the arguments to handle the case $d=2$.

\section{Iteration scheme}

\subsection{Choice of the parameters}\label{sec:param}
We define first the constant
\[
\gamma:= \Big(1+\frac 1p\Big) \left(\min \Big\{ \frac{d}{p}, \frac{d}{p'}, -1-d  \Big(\frac1{p'}-\frac1r\Big)\Big\}\right)^{-1}>0,
\]
where we have used crucially
\[
-1 - d \Big(\frac1{p'}-\frac1r\Big)
= d \left(\frac{1}{p}+\frac{1}{r} -1 - \frac{1}{d}\right) > 0\, .
\]
Notice that, up to enlarging $r$, we can assume that the quantity in the previous line is less than $1/2$, namely that $\gamma>2$.
Hence we set $\alpha:=4 + \gamma (d+1)$,
\begin{equation}
\label{eqn:choice-b}
b := \max\{ p, p'\} (3(1+\alpha)(d+2)+2),
\end{equation}
and 
\begin{equation}
\label{eqn:choice-beta}
\beta:= \frac 1{2 b} \min\Big\{ p, p', r, \frac{1}{b+1}\Big\}= \frac{1}{2 b(b+1)} .
\end{equation}
Finally, we choose $a_0$ and $M$ sufficiently large (possibly depending on all previously fixed parameters) to absorb numerical constants in the inequalities.
We set
\begin{equation}
\label{eqn:choice-ell}
\ell: = \lambda_{q}^{-1-\alpha},
\end{equation}
\begin{equation}
\label{eqn:choice-mu}\mu_{q+1} := \lambda_{q+1}^{\gamma}.
\end{equation}

\subsection{Convolution}\label{sec:convol}
We first perform a convolution of $\rho_q$ and $u_q$ to have estimates on more than one derivative of these objects and of the corresponding error. Let $\phi \in C^\infty_c(B_1)$ be a standard convolution kernel in space-time, $\ell$ as in \eqref{eqn:choice-ell} and define 
$$\rho_\ell := \rho_q \ast \phi_\ell, \qquad u_\ell := u_q \ast \phi_\ell, \qquad R_\ell := R_q \ast \phi_\ell. $$
We observe that $(\rho_\ell, u_\ell, R_\ell+ (\rho_q u_q)_\ell - \rho_\ell u_\ell)$ solves system \eqref{eqn:CE-R} 
and  by \eqref{eqn:ie-1}, \eqref{eqn:choice-beta} enjoys the following estimates

\begin{equation}
\label{eqn:r-l-l1}
\|R_\ell \|_{L^1} \leq \delta_{q+1},
\end{equation}
$$\| \rho_\ell - \rho_q\|_{L^p} \leq C \ell \|\rho_q\|_{C^1} \leq C \ell \lambda_q^\alpha C \leq C \delta_{q+1}^{1/p},$$
$$\| u_\ell - u_q\|_{L^{p'}} \leq C \ell \lambda_q^\alpha \leq C \delta_{q+1}^{1/{p'}},$$
$$\| u_\ell - u_q\|_{W^{1,r}} \leq C \ell \lambda_q^\alpha \leq C \delta_{q+1}^{1/{r}}\, .$$
Indeed note that by \eqref{eqn:choice-beta}
\[
\ell \lambda_q^\alpha = \lambda_q^{-1} = \delta_{q+1}^{\frac{1}{2b\beta}} \leq
\delta_{q+1}^{\max\{1/p, 1/p', 1/r\}}.
\]
Next observe that
\[
\|\partial_t^N \rho_\ell\|_{C^0} + \| \rho_\ell\|_{C^N} +\| u_\ell\|_{W^{1+N,r}} + \|\partial_t^N u_\ell\|_{W^{1,r}} 
\leq  C(N) \ell^{-N+1}( \| \rho_q\|_{C^1} +\| u_q\|_{W^{2,r}} 
)\leq  C(N) \ell^{-N+1}\lambda_{q}^\alpha
\]
for every $N\in \NN\setminus\{0\}$.
Using the Sobolev embedding $W^{d,r}\subset {W^{d,1}\subset }C^0$ we then conclude
\[
\|\partial_t^N u_\ell\|_{C^0} + \|u_\ell\|_{C^N} \leq C (N) \ell^{-N-d+2} \lambda_q^\alpha\, .
\]
By Young's inequality we estimate the higher derivatives of $R_\ell$ in terms of $\|R_q \|_{L^1}$ to get
\begin{equation}\label{e:D_tR}
\| R_\ell\|_{C^N} +\|\partial_t^N R_\ell\|_{C^0} 
\leq \|D^N \rho_\ell \|_{L^\infty} \|R_q \|_{L^1} \leq  C(N) \ell^{-N-d} \leq  C(N) \lambda_{q}^{(1+\alpha)(d+N)}\, 
\end{equation}
for every $N\in \NN$. 
Finally, for the last part of the error we show below that
\begin{equation}\label{e:commutatore}
\|(\rho_q u_q)_\ell - \rho_\ell u_\ell\|_{L^1} \leq C \ell^2 \lambda_q^{2\alpha} \leq \frac{1}{4} \delta_{q+2} ,
\end{equation}
(where we have assumed that $a$ is sufficiently large). The claim follows a well-known bilinear trick used often and originating (at least in the context of fluid dynamics) from the proof of Constantin, E and Titi of the positive part of the Onsager conjecture. We include a proof for the reader's convenience.

\begin{lemma} Consider a mollification kernel $\phi$ compactly supported in time and space. Then there is a constant $C = C(\phi)$ such that for every smooth functions $u$ and $\rho$ depending on time and space
	\begin{equation}
	\|(u\rho)*\phi_\ell - u*\phi_\ell\, \rho*\phi_\ell\|_{L^1} \leq C \ell^2\left(\|\partial_t u\|_{L^1} + \|D u\|_{L^1}\right) \left(\|D\rho\|_{C^0}+\|\partial_t \rho\|_{C^0}\right)\,.
	\end{equation} 
\end{lemma}
\begin{proof} To simplify our notation we introduce the variable $z=(x,t)$ and set $\Sigma:= (u\rho)*\phi_\ell - u*\phi_\ell\, \rho*\phi_\ell$. Assume without loss of generality that $\supp (\phi) \subset B^{d+1}_1\subseteq \RR^{d+1}$. 
	Simple computations lead to the formula
	\[
	\Sigma (z) := \frac{1}{2} \int\int (u (z-z') - u (z-z'')) (\rho (z-z') - \rho (z-z''))\phi_\ell (z')\phi_\ell (z'')\, dz' \, dz''\, ,
	\]
	which in turn implies
	\[
	|\Sigma (z)| \leq C \ell^{-2d-1} \|D_z\rho\|_{C^0} \int_{B^{d+1}_\ell}\int_{B^{d+1}_\ell}
	|u (z-z') - u (z-z'')|\, dz' \, dz''\, .
	\]
	Using $| u (z-z') - u (z-z'') |\leq |u (z-z') - u (z)| + |u (z) - u (z-z'')|$
	and integrating in the space variable $x$ we reach
	\begin{equation}\label{e:aggiunta}
	\|\Sigma (t, \cdot)\|_{L^1}\le C \ell^{-d} \|D_z\rho\|_{C^0} \int_{B^{d+1}_\ell} \int |u(t-t',x-x')-u(t,x)| \, dx\, dt'\, dx'\, .
	\end{equation}
	We then use $|u(t-t',x-x')-u(t,x)|\le |u(t-t',x-x')-u(t-t',x)|+|u(t-t',x)-u(t,x)|$ and
	estimate separately
	\begin{equation*}
	\int_{B^{d+1}_\ell} \int |u(t-t',x-x')-u(t-t',x)| \, dx\, dt'\, dx'
	\le C \ell^{d+2} \| Du\|_{L^1}
	\end{equation*}
	and
	\begin{equation*}
	\int_{B^{d+1}_\ell} \int |u(t-t',x)-u(t,x)| \, dx\, dt'\, dx'
	\le C \ell^{d+2} \| \partial_t u\|_{L^1}
	\end{equation*}
	Combining these last estimates with \eqref{e:aggiunta} we infer the desired conclusion. 
\end{proof}

\subsection{Definition of the perturbations}
\label{subsection:definitionofperturbations}
Let $\mu_{q+1}>0$ be as in \eqref{eqn:choice-mu} and let $\chi\in C^\infty_c (-\frac{3}{4}, \frac{3}{4})$ such that $\sum_{n \in \mathbb Z} \chi (\tau - n) =1$ for every $\tau\in \mathbb R$.
Let $\bar \chi\in C^{\infty}_c(-\frac{4}{5}, \frac{4}{5})$ be a nonnegative function satisfying $\bar \chi =1$ on $[-\frac{3}{4}, \frac{3}{4}]$.
Notice that $\sum_{n\in \mathbb{Z}} \bar \chi (\tau-n)\in [1,2]$ and  $\chi\cdot \bar \chi= \chi$.

Fix a parameter $\kappa=\frac{20
}{\delta_{q+2}}$ and consider two disjoint sets $\Lambda^1$, $\Lambda^2$ as in Lemma~\ref{lemma:geom}. Next, define $[i]$ to be $1$ or $2$ depending on the conrguence class of $i$. Finally, consider the building blocks introduced in Section~\ref{sec:build} in such a way that, for $\xi \in \cup_{i=1}^2\Lambda^i$, their spatial supports are disjoint.
We define the new density and vector field by adding to $\rho_\ell$ and $u_\ell$ a principal term and a smaller corrector, namely we set
\begin{align*}
\rho_{q+1} &:= \rho_\ell+ \theta_{q+1}^{(p)}+ \theta_{q+1}^{(c)}\,,\\ 
u_{q+1}  &:= u_\ell+ w_{q+1}^{(p)}+ w_{q+1}^{(c)}\, .
\end{align*}
The principal perturbations are given, respectively, by
\begin{align}
\label{eqn:theta}
w^{(p)}_{q+1} (t,x) &= \sum_{n \geq \nop}\bar  \chi (\kappa |R_\ell (t,x)| -n) \sum_{\xi\in \Lambda^{[n]}} W_{\xi , \mu_{q+1}, n /\kappa} (\lambda_{q+1} t, \lambda_{q+1} x),  \\ 
\label{eqn:w}
\theta^{(p)}_{q+1} (t,x) &= \sum_{n \ge \nop} \chi (\kappa |R_\ell (t,x)| - n) \sum_{\xi\in \Lambda^{[n]}}  a_{\xi}\left(\frac{R_\ell(t,x)}{|R_\ell(t,x)|}\right) \Theta_{\xi, \mu_{q+1}, n /\kappa} (\lambda_{q+1} t, \lambda_{q+1} x)\, ,
\end{align} 
where we understand that the terms in the second sum are all $0$ at points where $R_\ell$ vanishes. 
	In the definition of $w^{(p)}$ and $\theta^{(p)}$ the first sum runs for $n$ in the range 
	\begin{equation}
	\label{remark:sum is finite}
		12\leq n \leq \no C\ell^{-d}\delta_{q+2}^{-1} \le C \lambda_{q}^{d(1+\alpha) + 2\beta b^2}
	\le C \lambda_{q}^{d(1+\alpha)+1}.
	\end{equation}
	Indeed $\chi(\kappa |R_\ell(t,x)|-n)=0$ if $n\ge 20
	\delta_{q+2}^{-1}\|R_\ell\|_{C^0}+1$ and by \eqref{e:D_tR} we obtain an upper bound for $n$.

The aim of the corrector term for the density is to ensure that the overall addition has zero average:
\[
\theta^{(c)}_{q+1}:=-\int \theta_{q+1}^{(p)}(t,x)\, dx,
\]
The aim of the corrector term for the vector field is to ensure that the overall perturbation has zero divergence. Thanks to \eqref{eqn:avW}, we can apply Lemma~\ref{lemma23} to define

\[
w_{q+1}^{(c)}:=-\sum_{n \ge \nop} \sum_{\xi \in \Lambda^{[n]}} \mathcal{R} 
\left[ \nabla \bar \chi(\kappa |R_\ell(t,x)|-n) \cdot W_{\xi,\mu_{q+1},n/\kappa}(\lambda_{q+1} t, \lambda_{q+1} x) \right]
\]
Moreover, since $W_{\xi,\mu_{q+1},n/\kappa}$ is divergence-free, the argument inside $\mathcal{R}$ has $0$ average for every $t\geq 0$.

Notice finally that the perturbation equals $0$ on every time slice where $R_\ell$ vanishes identically.

\section{Proof of the Proposition~\ref{prop:inductive} in the case $d\geq 3$}

Before coming to the main arguments, we record some straightforward estimates for the ``slowly varying coefficients''.

\begin{lemma}\label{l:uglylemma}
	For $m\in \mathbb N$, $N\in \mathbb N\setminus \{0\}$ and $n\ge 2$ we have
	\begin{align}
	&\|\partial^m_t \chi (\kappa|R_\ell|-n)\|_{C^N}+\|\partial^m_t\bar \chi (\kappa|R_\ell|-n)\|_{C^N} \leq C (m, N) \delta_{q+2}^{-2(N+m)} \ell^{-(N+m)(1+d)} 
	\leq 
	C(m,N) \lambda_q^{(N+m) (d+2) (1+\alpha)}\\
	&\|\partial^m_t (a_\xi ({\textstyle{\frac{R_\ell}{|R_\ell|}}}))\|_{C^N}\leq C (m, N) \delta_{q+2}^{-N-m} \ell^{-(N+m)(1+d)} 
	\leq C(m,N)\lambda_q^{(N+m) (d+2) (1+\alpha)} \qquad\mbox{on $\{\chi (\kappa|R_\ell|-n) >0\}$}.
	\end{align}
\end{lemma}

\subsection{Estimate on $\|\theta_{q+1}\|_{L^p}$ and on $\inf_\TT \theta_{q+1}$}
We apply the improved H\"older inequality of \eqref{eqn:impr-holder},  Lemma \ref{l:uglylemma} and \eqref{remark:sum is finite} to get
\begin{equation}\label{eqn:theta-pert-p}
\begin{split}
\|\theta_{q+1}^{(p)}\|_{L^p} &\leq \sum_{n\ge \nop} \sum_{\xi\in \Lambda^{[n]}} \| \chi (\kappa |R_\ell (t,x)| - n)  \textstyle{a_{\xi}\left(\frac{R_\ell(t,x)}{|R_\ell(t,x)|}\right)} \|_{L^p} \|\Theta_{\xi, \mu_{q+1}, n /\kappa} (\lambda_{q+1} t, \lambda_{q+1} x) \|_{L^p}\\
\quad +& \frac{1}{\lambda_{q+1}^{1/p}} \sum_{n\ge \nop} \sum_{\xi\in \Lambda^{[n]}}  \| \chi (\kappa |R_\ell (t,x)| - n)  \textstyle{a_{\xi}\left(\frac{R_\ell(t,x)}{|R_\ell(t,x)|}\right)} \|_{C^1} \|\Theta_{\xi, \mu_{q+1}, n /\kappa} (\lambda_{q+1} t, \lambda_{q+1} x) \|_{L^p}
\\& \le C\sum_{n\ge \nop} \| (n/k)^{1/p} \chi (\kappa |R_\ell (t,x)| - n)  \|_{L^p} 
+C \lambda_{q+1}^{-1/p} \delta_{q+2}^{1/p} \lambda_q^{(d+2)(1+\alpha)+(1+1/p)((d(1+\alpha)+1))}
\\& \le C  \|R_\ell\|_{L^1}^{1/p} + C \lambda_{q+1}^{-1/p} \delta_{q+2}^{1/p} \lambda_q^{3(d+2) (1+\alpha)} 
\\& 
\leq C \delta_{q+1}^{1/p},
\end{split}
\end{equation}
provided that in the second last inequality we use $(d+2)(1+\alpha)+(1+1/p)((d(1+\alpha)+1))\le 3(1+\alpha)(d+2)$ and in the last inequality we use \eqref{eqn:r-l-l1}.

Next, by means of \eqref{eqn:l26} applied to the $\lambda_{q+1}^{-1}$-periodic function $(\Theta_{\xi, \mu_{q+1}, n/\kappa}(\lambda_{q+1}t, \lambda_{q+1}x)$, by \eqref{eqn:Thetanorms} (precisely $\| \Theta_{\xi, \mu_{q+1}, n/ \kappa}\|_{L^1}\le (\frac{n}{\kappa})^{1/p}\mu_{q+1}^{-d/p'}$), Lemma \ref{l:uglylemma} and \eqref{remark:sum is finite}, we estimate
\begin{align*}
| \theta_{q+1}^{(c)} (t)| 
& \le C\lambda_{q+1}^{-1} \sum_{n\ge \nop}\sum_{\xi\in \Lambda^{[n]}} \|  \chi( \kappa|R_\ell| - n ) \textstyle{a_{\xi}\left( \frac{ R_\ell }{ |R_\ell| } \right)} \|_{C^1}
\| \Theta_{\xi, \mu_{q+1}, n/\kappa}\|_{L^1}
\\&\qquad + \sum_{n\ge \nop}\sum_{\xi\in \Lambda^{[n]}} \| \chi( \kappa|R_\ell| - n ) \textstyle{a_{\xi}\left( \frac{ R_\ell }{ |R_\ell| } \right)} \|_{L^1} \| \Theta_{\xi, \mu_{q+1}, n/\kappa}\|_{L^1}
\\& \le C\lambda_{q+1}^{-1} \mu_{q+1}^{-d/p'} \lambda_{q}^{3(1+\alpha)(d+2)}
+ C \mu_{q+1}^{-d/p'} \sum_{n\ge \nop} \| \chi( (n/\kappa)^{1/p} \kappa|R_\ell| - n ) \textstyle{a_{\xi}\left( \frac{ R_\ell }{ |R_\ell| } \right)} \|_{L^1} 
\\& \le \frac{1}{2}\delta_{q+1}^{1/p} + \mu_{q+1}^{-d/p'}\|R_\ell\|_{L^1}^{1/p}\le \delta_{q+1}^{1/p}.
\end{align*} 
From the latter inequality, since $\theta_{q+1}^{(p)}$ is nonnegative, we also deduce that
\[
\inf_\TT \theta_{q+1}(t) \geq - \theta_{q+1}^{(c)}(t) \geq - \delta_{q+1}^{1/p} \, ,
\]
namely Statement (b) of Proposition~\ref{prop:inductive}.

\subsection{Estimate on $\|w_{q+1}\|_{L^{p'}}$ and on $\|Dw_{q+1}\|_{L^{r}}$}\label{sec:end}
Exactly with the same computation as in \eqref{eqn:theta-pert-p}, replacing $p$ with $p'$, we have that
\begin{equation}
\|w_{q+1}^{(p)}\|_{L^{p'}} 
\leq C \|R_\ell\|_{L^1}^{1/p'}+
C \lambda_{q+1}^{-1/p'} \delta_{q+2}^{1/p'} \lambda_q^{3(d+2) (1+\alpha)}
\leq C \delta_{q+1}^{1/p'}
\end{equation}
Concerning the corrector term $w_{q+1}^{(c)}$, we use \eqref{eqn:Wnorms} (precisely $\| W_{\xi, \mu_{q+1}, n/ \kappa}\|_{L^{p'}}\le (\frac{n}{\kappa})^{1/p'}\le \lambda_{q}^{(d+2)(1+\alpha)}
$) Lemma~\ref{lemma23} and \eqref{remark:sum is finite} to get
\begin{equation}\label{eqn:wcpprime}
\begin{split}
\|w_{q+1}^{(c)}\|_{L^{p'}}  &\leq
\frac 1 {\lambda_{q+1}} \sum_{n\ge \nop}\sum_{\xi \in \Lambda^{[n]}} 
\|\bar \chi(\kappa|R_\ell|-n)\|_{C^2}\|W_{\xi,\mu_{q+1},n/\kappa}\|_{L^{p'}}
\\& \leq  C\lambda_{q+1}^{-1}\sum_{n\geq\nop}
 \lambda_{q}^{2(d+2)(1+\alpha)} (n/\kappa)^{1/p'}
\\&\leq C\lambda_{q+1}^{-1} \delta_{q+2}^{1/p'}\lambda_{q}^{4(d+2)(1+\alpha)}
\le \delta_{q+2}^{1/p'} \le \delta_{q+1}^{1/p'}.
\end{split}
\end{equation}
Computing the gradient of $w_{q+1}^{(p)}$ and combining Lemma \ref{l:uglylemma} with \eqref{eqn:Wnorms} we have
\begin{align}
\|Dw_{q+1}^{(p)}\|_{L^r} &\leq  
\sum_{n\ge \nop}\sum_{\xi\in \Lambda^{[n]}} \|\bar \chi(\kappa |R_\ell|-n)\|_{C^1} \| W_{\xi,\mu_{q+1},n/\kappa}\|_{L^r}
+\sum_{n\ge \nop}\sum_{\xi\in \Lambda^{[n]}} \lambda_{q+1} \|DW_{\xi,\mu_{q+1},n/\kappa}\|_{L^r}
\\& \leq C \delta_{q+2}^{1/p'} \lambda_{q}^{3(1+\alpha)(d+2)+2+b\gamma d(1/p'-1/r)}
+C \delta_{q+2}^{1/p'} \lambda_{q}^{b+3(1+\alpha)(d+2)+b\gamma(1+d(1/p'-1/r))}
\\& \leq  \delta_{q+2}^{1/p'}\le \delta_{q+1}^{1/r}.
\end{align}
Concerning the corrector, by Lemma~\ref{lemma23} and similar computations as above,
\begin{align}
\|Dw_{q+1}^{(c)}\|_{L^r} &\leq C \sum_{n\ge \nop}\sum_{\xi\in \Lambda^{[n]}} \| \bar\chi(\kappa |R_\ell|-n)\|_{C^3} \|W_{\xi, \mu_{q+1}, n/\kappa}\|_{L^r} 
\\& \leq C \delta_{q+2}^{1/p'} \lambda_q^{ 5(1+\alpha)(d+2)
} \mu_{q+1}^{d(\frac 1 r - \frac 1 {p'})}
\leq \delta_{q+2}^{1/p'} \lambda_{q}^{5(1+\alpha)(d+2)-b(1+1/p)}
\leq \delta_{q+1}^{1/r}.
\end{align}

\subsection{New error $R_{q+1}$}\label{sec:newerror}
By definition the new error $R_{q+1}$ must satisfy
\begin{equation}\label{eqn:new-error}
\begin{split}
-\div R_{q+1} =& \partial_t \rho_{q+1} + \div(  \rho_{q+1} u_{q+1}) = \div ( \theta_{q+1}^{(p)} w_{q+1}^{(p)} - R_\ell) + \partial_t \theta_{q+1}^{(p)}+\partial_t \theta_{q+1}^{(c)} 
\\&+ \div( \theta_{q+1}^{(p)} u_\ell+ \rho_\ell w_{q+1} + \theta_{q+1}^{(p)}w_{q+1}^{(c)}
) + \div ( (\rho_q u_q)_\ell - \rho_\ell u_\ell) 
\end{split}
\end{equation}
In the second equality above we have used that $(\rho_\ell, u_\ell, R_\ell+ (\rho_q u_q)_\ell - \rho_\ell u_\ell)$ solves \eqref{eqn:CE-R}, that $\div u_\ell = \div w_{q+1}=0$, and that $\theta_{q+1}^{(c)}$ is constant in space.

Let us write
\begin{align*}
\partial_t \theta_{q+1}^{(p)} & = \sum_{n\ge \nop}\sum_{\xi\in \Lambda^{[n]}} \chi(\kappa |R_\ell|- n) \textstyle{a_{\xi}\left( \frac{R_\ell}{|R_\ell|} \right)} \partial_t\left[ \Theta_{\xi, \mu_{q+1}, n /\kappa}(\lambda_{q+1} t, \lambda_{q+1}x) \right]
\\& \quad + \sum_{n\ge \nop}\sum_{\xi\in \Lambda^{[n]}} \partial_t \left[ \chi(\kappa |R_\ell|- n) \textstyle{a_{\xi}\left( \frac{R_\ell}{|R_\ell|} \right)}\right]  \Theta_{\xi, \mu_{q+1}, n /\kappa}(\lambda_{q+1} t, \lambda_{q+1}x)
\\&=: ( \partial_t \theta^{(p)}_{q+1})_1+(\partial_t \theta^{(p)}_{q+1})_2,
\end{align*}
by using that $\Theta_{\xi, \mu_{q+1}, n /\kappa}$ and $W_{\xi, \mu_{q+1}, n /\kappa}$ solve the transport equation \eqref{eqn:itsolves} and Lemma~\ref{lemma:geom} we get the cancellation of the error $R_\ell$ up to lower order terms
\begin{align}\label{eqn:formulone}
\div &( \theta_{q+1}^{(p)} w_{q+1}^{(p)}) + (\partial_t \theta_{q+1}^{(p)})_1 - \div R_\ell
\\& =
\sum_{n\ge \nop} \sum_{\xi\in \Lambda^{[n]}}   \nabla \left[\chi(\kappa |R_\ell|- n)
\textstyle{a_{\xi}\left(\frac{R_\ell}{|R_\ell|}\right)}\right] (\Theta_{\xi, \mu_{q+1}, n /\kappa} W_{\xi, \mu_{q+1}, n /\kappa})(\lambda_{q+1} t, \lambda_{q+1}x)- \div R_\ell
\\&+\sum_{n\ge \nop} \sum_{\xi\in \Lambda^{[n]}} \chi(\kappa |R_\ell|- n)
\textstyle{a_{\xi}\left(\frac{R_\ell}{|R_\ell|}\right)}
\lambda_{q+1} \left[ \partial_t \Theta_{\xi, \mu_{q+1}, n /\kappa}+ \div(\Theta_{\xi, \mu_{q+1}, n /\kappa}W_{\xi, \mu_{q+1}, n /\kappa})\right](\lambda_{q+1} t, \lambda_{q+1}x)
\\&=\sum_{n\ge \nop} \sum_{\xi\in \Lambda^{[n]}}   \nabla \left[\chi(\kappa |R_\ell|- n)
\textstyle{a_{\xi}\left(\frac{R_\ell}{|R_\ell|}\right)}\right] \left[(\Theta_{\xi, \mu_{q+1}, n /\kappa} W_{\xi, \mu_{q+1}, n /\kappa})(\lambda_{q+1} t, \lambda_{q+1}x)- \frac{n}{\kappa} \xi \right]
\\&+ \sum_{n\ge \nop} \sum_{\xi\in \Lambda^{[n]}}   \nabla \left[\chi(\kappa |R_\ell|- n)
\textstyle{a_{\xi}\left(\frac{R_\ell}{|R_\ell|}\right)}\right] \frac{n}{\kappa}\xi - \div R_\ell
\\& =\sum_{n\ge \nop} \sum_{\xi\in \Lambda^{[n]}}   \nabla \left[\chi(\kappa |R_\ell|- n)
\textstyle{a_{\xi}\left(\frac{R_\ell}{|R_\ell|}\right)}\right] \left[(\Theta_{\xi, \mu_{q+1}, n /\kappa} W_{\xi, \mu_{q+1}, n /\kappa})(\lambda_{q+1} t, \lambda_{q+1}x)- \frac{n}{\kappa} \xi \right]
+ \div (\tilde R_\ell-R_\ell),
\end{align}
where
\[
\tilde R_\ell:= 
\sum_{n\ge \nop} \chi(\kappa |R_\ell|- n)	\frac{R_\ell}{|R_\ell|} \frac{n}{k}.
\]
We have
\begin{align}\label{eqn:tildeRestimate}
\notag
|R_\ell-\tilde R_\ell| & \le \Big|\sum_{n=-1}^{11
} \chi(\kappa |R_\ell|- n) R_\ell \Big|
+\Big| \sum_{n\ge \nop} \chi(\kappa |R_\ell|- n)\left( \frac{R_\ell}{|R_\ell|}\frac{n}{k}-R_\ell\right)	\Big|   
\\& \le \frac{13
}{\kappa} + \sum_{n\ge \nop} \chi(\kappa |R_\ell|- n)\left| |R_\ell|-\frac{n}{\kappa}\right|
\\& \le \frac{13
}{20
}\delta_{q+2}+\frac{3}{40
}\delta_{q+2}
\le \frac{15
}{20
} \delta_{q+2}.
\end{align}
We can now define $R_{q+1}$ which satisfies \eqref{eqn:new-error} as 
\begin{equation}
\begin{split}
-R_{q+1} :=& R^{quadr}+ (\tilde R_\ell- R_\ell) + R^{time}  +\theta_{q+1}^{(p)} u_\ell+ \rho_\ell w_{q+1} + \theta_{q+1}^{(p)}w_{q+1}^{(c)}+[(\rho_q u_q)_\ell - \rho_\ell u_\ell],
\end{split}
\end{equation}
where
\begin{equation}
\label{defn:Rquadr}
R^{quadr}: = \sum_{n\ge \nop} \sum_{\xi\in \Lambda^{[n]}} \fR \left[  \nabla \left(\chi(\kappa |R_\ell|- n)
\textstyle{a_{\xi}\left(\frac{R_\ell}{|R_\ell|}\right)}\right)\cdot  \left((\Theta_{\xi, \mu_{q+1}, n /\kappa} W_{\xi, \mu_{q+1}, n /\kappa})(\lambda_{q+1} t, \lambda_{q+1}x)- \frac{n}{\kappa} \xi \right)\right],
\end{equation}
\begin{equation}
\label{eqn:Rc}
R^{time}:=\nabla \Delta^{-1}( (\partial_t \theta_{q+1}^{(p)})_2 +\partial_t \theta_{q+1}^{(c)} + m ),
\end{equation}
\begin{equation*}
m :=\sum_{n\ge \nop} \sum_{\xi\in \Lambda^{[n]}} \int    \nabla \left[\chi(\kappa |R_\ell|- n)
\textstyle{a_{\xi}\left(\frac{R_\ell}{|R_\ell|}\right)}\right] \left[(\Theta_{\xi, \mu_{q+1}, n /\kappa} W_{\xi, \mu_{q+1}, n /\kappa})(\lambda_{q+1} t, \lambda_{q+1}x)- \frac{n}{\kappa} \xi \right] \, dx.
\end{equation*}
Notice that $R^{quadr}$ is well defined since by \eqref{eqn:rightaverage} the function $(\Theta_{\xi, \mu_{q+1}, n /\kappa} W_{\xi, \mu_{q+1}, n /\kappa})(\lambda_{q+1} t, \lambda_{q+1}x)- \frac{n}{\kappa} \xi$ has $0$ mean.
We now estimate in $L^1$ each term in the definition of $R_{q+1}$. From the second equality in \eqref{eqn:new-error} and since the average of $(\partial_t \theta_{q+1}^{(p)})_2$ is $m$ by integration by parts, we deduce that $ (\partial_t \theta_{q+1}^{(p)})_2 +\partial_t \theta_{q+1}^{(c)} + m$ has $0$ mean, so that $R^{time}$ is well defined.
Recall that  the estimate on $\|(\rho_q u_q)_\ell - \rho_\ell u_\ell\|_{L^1}$ has been already established in \eqref{e:commutatore}. 

By the property \eqref{ts:antidiv} of the antidivergence operator $\fR$, Lemma \ref{l:uglylemma} and \eqref{remark:sum is finite} we have
\begin{align*}
\|R^{quadr}\|_{L^1}  & \leq \frac C {\lambda_{q+1}} \sum_{n\ge \nop}\sum_{\xi\in \Lambda^{[n]}}
\| \chi(\kappa |R_\ell|- n)
\textstyle{a_{\xi}\left(\frac{R_\ell}{|R_\ell|}\right)}\|_{C^2}\|\Theta_{\xi, \mu_{q+1}, n/\kappa} W_{\xi, \mu_{q+1}, n/\kappa}\|_{L^1} \\&\leq C \delta_{q+2} \frac{\lambda_q^{4(1+\alpha)(d+2)+2}}{\lambda_{q+1}} \leq \frac{\delta_{q+2}}{20}.
\end{align*}
To estimate the terms which are linear with respect to the fast variables, we take advantage of the concentration parameter $\mu_{q+1}$. First of all, by Calderon-Zygmund estimates we get
\begin{align*}
\| R^{time} \|_{L^1} \le  C \| (\partial_t \theta_{q+1}^{(p)})_2+\partial_t \theta_{q+1}^{(c)} - m \|_{L^1}
\le  \| (\partial_t \theta_{q+1}^{(p)})_2 \|_{L^1}+ | \partial_t \theta_{q+1}^{(c)} | + | m |	.
\end{align*}
Next, notice that
\begin{align}\label{eqn:tpl1}
\| (\partial_t \theta_{q+1}^{(p)})_2 \|_{L^1}
& \leq 
C \sum_{n\ge \nop}\sum_{\xi\in \Lambda^{[n]}}\| \partial_t\big[ \chi(\kappa |R_\ell|- n) \textstyle{a_{\xi}\left( \frac{R_\ell}{|R_\ell|} \right)}\big]\|_{C^0}
\|\Theta_{\xi, \mu_{q+1}, n /\kappa} \|_{L^1}
\\& \leq C \delta_{q+2}^{1/p} \lambda_{q}^{3(1+\alpha)(d+2)}\mu_{q+1}^{-d/p'} \leq \frac{\delta_{q+2}}{20}.
\end{align}
From \eqref{eqn:tpl1}, \eqref{eqn:itsolves} and \eqref{eqn:l26} we get
\begin{align*}
& | \partial_t \theta_{q+1}^{(c)} |  + | m |	
\\ & \le \| (\partial_t \theta_{q+1}^{(p)})_2 \|_{L^1}
+ \left| \sum_{n\ge \nop}\sum_{\xi\in \Lambda^{[n]}} \int \chi(\kappa|R_\ell| - n ) \textstyle{ a_{\xi} \left( \frac{ R_\ell }{|R_\ell|} \right) }	\partial_t \left[ \Theta_{\xi, \mu_{q+1}, n /\kappa}(\lambda_{q+1} t, \lambda_{q+1}x 
) \right]\, dx  \right| + |m|
\\& \le \frac{\delta_{q+2}}{20} + 
\left| \sum_{n\ge \nop}\sum_{\xi\in \Lambda^{[n]}} \int \chi(\kappa|R_\ell| - n ) \textstyle{ a_{\xi} \left( \frac{ R_\ell }{|R_\ell|} \right) }	\div \left[ (\Theta_{\xi, \mu_{q+1}, n /\kappa}W_{\xi, \mu_{q+1}, n/\kappa})(\lambda_{q+1} t, \lambda_{q+1}x 
) \right]\, dx  \right| + |m|
\\&=\frac{\delta_{q+2}}{20} +
2\left| \sum_{n\ge \nop}\sum_{\xi\in \Lambda^{[n]}} \int \nabla \left[\chi(\kappa|R_\ell| - n ) \textstyle{ a_{\xi} \left( \frac{ R_\ell }{|R_\ell|} \right) }\right]\cdot  \left[ (\Theta_{\xi, \mu_{q+1}, n /\kappa}W_{\xi, \mu_{q+1}, n/\kappa})(\lambda_{q+1} t, \lambda_{q+1}x 
)-\frac{n}{k}\xi \right]\, dx  \right|
\\& \le \frac{\delta_{q+2}}{20} + 
\frac{2 \sqrt{d}}{\lambda_{q+1}}\sum_{n\ge \nop}\sum_{\xi\in \Lambda^{[n]}} \| \chi(\kappa|R_\ell| - n ) \textstyle{ a_{\xi} \left( \frac{ R_\ell }{|R_\ell|} \right) }\|_{C^2} \| \Theta_{\xi, \mu_{q+1}, n /\kappa}W_{\xi, \mu_{q+1}, n/\kappa}\|_{L^1}
\\& \le \frac{\delta_{q+2}}{20}+ C \lambda_{q+1}^{-1} \delta_{q+2} \lambda_{q}^{4(1+\alpha)(d+2)}
\le \frac{1}{10}\delta_{q+2}.
\end{align*} 
Similarly, we have that
\begin{equation}
\begin{split}
\| & \theta_{q+1}^{(p)} u_\ell + \rho_\ell w_{q+1}^{(p)}\|_{L^1} \leq \| \theta_{q+1}^{(p)}\|_{L^1} \| u_\ell\|_{L^\infty} + \|\rho_\ell\|_{L^\infty} \|w_{q+1}^{(p)}\|_{L^1}
\\&\leq  \sum_{n\ge \no} \sum_{\xi \in \Lambda^{[n]}} \| \chi(\kappa|R_\ell| - n ) \textstyle{ a_{\xi} \left( \frac{ R_\ell }{|R_\ell|} \right) }\|_{L^\infty} \| \Theta_{\xi, \mu_{q+1},n/\kappa}\|_{L^1} \| u_\ell\|_{L^\infty} + \|\rho_\ell\|_{L^\infty} \|\bar \chi(\kappa|R_\ell| - n)\|_{L^{\infty}}\|W_{\xi,\mu_{q+1},n/\kappa}\|_{L^1}
\\& \leq C\delta_{q+2}^{1/p} \lambda_{q}^{2(1+\alpha)(d+2)} \mu^{-d/p'}_{q+1} + C\delta_{q+2}^{1/p'} \lambda_{q}^{2(1+\alpha)(d+2)}\mu^{-d/p}_{q+1}
\leq \frac{\delta_{q+2}}{20}.
\end{split}
\end{equation}
In the last inequality we used  $2\beta b^2\le 1$, the definition of $\gamma$, and  $b(1+1/p)\ge 2(1+\alpha)(d+2)+1$.

Finally, from \eqref{eqn:theta-pert-p} and \eqref{eqn:wcpprime}
\begin{equation}
\begin{split}
\|  ({\rho_{\ell}}+ \theta_{q+1}^{(p)})w_{q+1} ^{(c)}
\|_{L^1} &\leq
(\|{\rho_{\ell}}\|_{C^1}+ \|  \theta_{q+1}^{(p)}\|_{L^{p}}) \|w_{q+1}^{(c)}\|_{L^{p'}}
\\&{\leq C \lambda_q^{4(1+\alpha)(d+2)+\alpha} \lambda_{q+1}^{-1}}
\leq \frac{1}{20} \delta_{q+2}
.
\end{split}
\end{equation}

\subsection{Estimates on higher derivatives}\label{sec:highderiv} By the choice of $\alpha$, since in particular $\alpha\ge 2+\gamma(d+1)$, we have that
\begin{equation}
\label{eqn:est-rhoC1}
\begin{split}
\| \rho_{q+1}\|_{C^1} &\leq \| \rho_{\ell}\|_{C^1}+ \| \theta_{q+1}\|_{C^1} \leq \| \rho_{q}\|_{C^1}
+ \sum_{n\ge \no} \sum_{\xi\in \Lambda^{[n]}} \| \chi(\kappa|R_\ell| - n ) \textstyle{ a_{\xi} \left( \frac{ R_\ell }{|R_\ell|}\right)}\|_{C^1} \|\Theta_{\xi, \mu_{q+1},n/\kappa}(\lambda_{q+1}x) \|_{C^1}
\\&\leq C \lambda_q^\alpha + 
C \lambda_q^{3(1+\alpha)(d+2)} \lambda_{q+1}\mu_{q+1}^{1+d/p}
\leq \lambda_{q+1}^{\alpha}.
\end{split}
\end{equation}
An entirely similar estimate is valid for $\|\partial_t \rho_{q+1}\|_{C^0}$
and the one for $\|u_\ell+ w_{q+1}^{(p)}\|_{W^{2,r}}$ is analogous. 
Concerning $ w_{q+1}^{(c)}$, we use Lemma~\ref{lemma23} and \eqref{remark:sum is finite}
\begin{equation*}
\begin{split}
\| w_{q+1}^{(c)}\|_{W^{2,r}} &\le 
\sum_{n\ge \nop}\sum_{\xi\in \Lambda^{[n]}} \lambda_{q+1}^{}
\| \bar  \chi(\kappa |R_\ell|-n) \|_{C^4} \|W_{\xi , \mu_{q+1}, n/\kappa} \|_{W^{2,r}}
\\&\leq C \lambda_q^{6 (1+\alpha) (d+2)} \lambda_{q+1}^2\mu_{q+1}^{2+ d(1/p'-1/r)}
\leq  \lambda_{q+1}^{\alpha}.
\end{split}
\end{equation*}
It remains just to estimate
\[
\| \partial_t u_q \|_{L^1}   \le  \| \partial_t u_\ell \|_{L^1}+\|\partial_t w^{(p)}_{q+1}\|_{L^1} + \|\partial_t w^{(c)}_{q+1}\|_{L^1}.
\]
From \eqref{eqn:Wnorms} and Lemma~\ref{l:uglylemma}
\begin{align*}
\|\partial_t w^{(p)}_{q+1} \|_{L^1} & \le \sum_{n\ge \nop}\sum_{\xi\in \Lambda^{[n]}} \lambda_{q+1}\| \partial_t W_{\xi, \mu_{q+1}, n/\kappa} \|_{L^1}
+ \| \partial_t \bar{\chi} (\kappa |R_\ell| - n )  \|_{L^{\infty}}  \|  W_{\xi, \mu_{q+1}, \kappa/n}  \|_{L^1}
\\& C\delta_{q+2}^{2/p'}\lambda_{q}^{(1+2/p')(d(1+\alpha)+1)}\lambda_{q+1}\mu_{q+1}^{1+\gamma(1+d(2/p'-1))}\le \lambda_{q+1}^{2+\gamma(d+1)}\le \lambda_{q+1}^{\alpha}.
\end{align*}
A similar computation is valid for $\| \partial_tw_{q+1}^{(c)} \|_{L^1}$.

\section{The building blocks and the iterative scheme in dimension $d=2$}\label{s:2d}
We describe in this section how the proof of Proposition~\ref{prop:inductive}, concluded above in dimension $d \geq 3$, should be modified to cover the case $d=2$. The main obstruction in this regard is that the building blocks in Section 4 cannot be translated, as we did in dimension $d\geq 3$ in Lemma \ref{lemma:disjointsupports}, to make sure that their support is disjoint in space. 
In dimension $d=2$, we need to make them disjoint in space-time, observing that they are described by a small ball which translates at constant speed according to a translating vector field which is supported on a slightly larger ball.

\subsection{Iteration scheme and definition of the perturbations}
We choose the parameters as in Section~\ref{sec:param} and we perform the convolution step as in Section~\ref{sec:convol}. Similarly, the cutoffs $\chi\in C^\infty_c (-\frac{3}{4}, \frac{3}{4})$ and $\bar \chi\in C^{\infty}_c(-\frac{4}{5}, \frac{4}{5})$ are chosen as in Section~\ref{subsection:definitionofperturbations}, as well as $\kappa=\frac{20
}{\delta_{q+2}}$ and the sets $\Lambda^1$, $\Lambda^2$. 

Starting from the building blocks introduced in Section~\ref{sec:build} we will choose positive real numbers $v_n$ (which will satisfy $|v_n-n|\leq 1$) and real numbers $a_{\xi,n}$ and define
\begin{equation}
	W_{\xi, \mu_{q+1}, v_n/\kappa}( t , x ):=\tilde W_{\xi, \mu_{q+1}, v_n/\kappa}( t, x-a_{\xi,n}\xi),
	\quad
	\Theta_{\xi, \mu_{q+1}, v_n/\kappa}(t, x):=\tilde \Theta_{\xi, \mu_{q+1}, v_n/\kappa}(t, x-a_{\xi,n} \xi),
\end{equation}
for any $n\ge 1$ and $\xi\in \Lambda^{[n]}$. The difficult part will be to choose $v_n$ and $a_{\xi,n}$ so that 

\begin{equation}\label{eq:disjointsupportsdim2}
	W_{\xi, \mu_{q+1}, v_n/\kappa} \cdot \Theta_{\xi', \mu_{q+1}, v_m/\kappa}( t , x )=0
	\quad \text{for any $(x,t)\in \TT^2\times \RR^+$,}
\end{equation}
whenever $\xi\neq \xi'$, $|n-m|\le 1$ and $n,m\le  C  \lambda_{q}^{d(1+\alpha)+1}$.

Assuming for the moment that this can be done, we define the new density and vector field as we did in Section \ref{subsection:definitionofperturbations} up to replacing all functions $\Theta_{\xi, \mu_{q+1}, n /\kappa}$ and $W_{\xi, \mu_{q+1}, n /\kappa}$ with $\Theta_{\xi, \mu_{q+1}, v_n /\kappa}$ and $W_{\xi, \mu_{q+1}, v_n /\kappa}$.
The proof of the proposition would then follow the same arguments: we only need to modify sligthly the definition of $R_{q+1}$ . Most of this section will be devoted to choose $v_n$ and $a_{\xi,n}$ so that \eqref{eq:disjointsupportsdim2} holds. Once we have achieved the latter, we will then show how to change the definition of $R_{q+1}$.

\subsection{Geometric arrangement} The main geometric construction is given by the following proposition. $\Lambda_1 \cup \Lambda_2$ is the set of possible space directions for the building blocks, while the sequence $\{w_n\}$ is in fact the set of values $ \mu_{q+1}^{d/p'} \left(\frac{n}{\kappa}\right)^{1/p'}$. Observe that ,when $\frac{w_n}{w_{n-1}}$ is a rational number, the relative position of the space supports of the building blocks is time-periodic. If each space support were merely a point we could easily make them always disjoint and in fact we could identify their minimum distance. If we write $\frac{w_n}{w_{n-1}} = 1 + \frac{A(n)}{N(n)}$ with $A(n)$ and $N(n)$, intuitively such minimum distance should be made comparable to $\frac{1}{N(n)}$. 

\begin{propos}\label{prop:disjointsupportdim2}
	Consider two disjoint sets $\Lambda^1, \Lambda^2 \subset \RR^2$ as in Lemma \ref{lemma:geom}. Let $\{w_n\}_{n\in {\NN\setminus \{0\}}}\subset \RR$  satisfy
	\begin{equation*}
		\frac{w_n}{w_{n-1}}= 1 + \frac{A(n)}{N(n)} < 10
		\quad \text{for every $n\ge 2$}
	\end{equation*}
	where $N(n)\le \bar C n$, for a given $\bar C>0$. 
	
	Then there exists a constant $c_0:= c_0(\bar C, \Lambda^1,\Lambda^2) >0$ with the following property. For every $\xi \in \Lambda^k$ and $n\in \NN$ there exists 
	$a_{\xi,n}\in [0,1]$ such that the family of curves
	\begin{equation}
		x_{\xi,n}(t):= (w_n t+a_{\xi,n}) \xi \qquad \mbox{with $\xi\in \Lambda^k$,  $n\equiv k$ mod $2$ and $k\in \{1,2\}$}
	\end{equation}
	satisfies
	\begin{equation}
		d_{\TT^2}(x_{\xi,n}(t) , x_{\xi',m}(t)) \ge \frac{c_0}{n}
		\quad \text{for every $t\ge 0$, when $|n-m|\le 1$ and $\xi\neq \xi'$}.
	\end{equation}
\end{propos}

The proof of the proposition is based on the following elementary lemma.

\begin{lemma}\label{lemma:disjointsupportdim2}
	Fix two different vectors $\xi, \xi'\in \mathbb{S}^{1}\cap \mathbb{Q}^2$ and a number $w= 1+ \frac{A}{N}<10$, with $A$ and $N$ positive integers and coprime. Then there exists $C=C(\xi, \xi')$ such that
	\begin{itemize}
		\item[(i)]
			$\Leb 1( [0,1] \setminus  \{s : \, d_{\TT^2}( t\xi, (t + s) \xi')\ge \eps \quad \forall\, t\ge 0  \})< C \eps$;
		\item[(ii)]
		$\Leb 1( [0,1] \setminus \{s: \, d_{\TT^2}( t \xi, (t w + s)  \xi')\ge \eps N^{-1} \quad \forall\, t\ge 0  \})< C \eps$.
	\end{itemize}
\end{lemma}

\begin{proof}
	Set $T_{int}:=\{ (t,t'): \, \xi t= \xi' t' \, \text{on} \, \TT^2 \}$ and observe that $T_{int}\subset \mathbb{Q}^2$ since the matrix with columns $\xi$ and $\xi'$ is invertible  with rational coefficients. Moreover $T_{int}$ is an additive discrete subgroup of $\RR^2$, hence it is a free group of rank $k\in \{0,1,2\}$.
    Denoting by $T$ and $T'$ the period of, respectively, $t\to \xi t$ and $t\to \xi' t$ one has that $(T,0),(0,T')\in T_{int}$. This implies that the rank of $T_{int}$ is two, hence we can find two generators $(t_1, t'_1), (t_2, t_2')\in T_{int}$.
    
    Let us finally introduce $A:=\{ \xi t\in \TT^2: \, (t,s)\in T_{int}\, \text{ for some } s\in \RR \}$ to denote the set of points in $\TT^2$ where the supports of the curves $t\to t \xi$ and $t\to t\xi'$ intersect. 
	
	\medskip
	
	Let us now prove (i). Let $s\in [0, 1]$ be such that there exists $t\ge 0$ satisfying $d_{\TT^2}( t\xi, (t+s)\xi')<\eps$.
	There exists $q\in A$ such that
	$d_{\TT^2}(t\xi, q)\le \bar c\eps$, where $\bar c=\bar c(\xi,\xi')>1$,
	hence up to modifying $t$ we can assume that $t\xi=:q \in A$ and $d_{\TT^2}(q,(t+s) \xi')\le 2\bar c\eps$. Since $t\xi\in A$ there exists $t'$ such that $(t,t')\in T_{int}$ and, exploiting the fact that $(t_1, t'_1), (t_2, t_2')\in T_{int}$ are generators, we can find $k_1, k_2\in \mathbb{Z}$ such that $t= k_1 t_1+k_2 t_2$ and $t'=k_1 t_1' + k_2 t_2'$.  The following identity holds on $\TT^2$ 
	\begin{equation*}
		(t+s)\xi' = (k_1 t_1+k_2 t_2) \xi' + s \xi' = (k_1(t_1-t_1')+k_2(t_2-t_2'))\xi' + q + s\xi',
	\end{equation*}
	therefore $d_{\TT^2}((k_1(t_1-t_1')+k_2(t_2-t_2')+s)\xi', 0)\le 2\bar c \eps$. This implies $-s\in B_{2\bar c\eps}(k_1(t_1-t_1')+k_2(t_2-t_2')) + \mathbb{Z}T'$.
    Notice that the set $E:=\{  k_1(t_1-t_1')+k_2(t_2-t_2'): \, k_1, k_2\in \mathbb{Z} \}$ is discrete since $t_1-t_1'$, $t_2-t_2'$ are rational numbers. Any two consecutive points in $E$ have a fixed distance $c=c(\xi, \xi')>0$ and $E+\mathbb{Z}T'=E$. In particular
    \begin{equation}
        \Leb 1 ( [0,1] \setminus  \{s : \, d_{\TT^2}( t\xi, (t + s) \xi')\ge \eps \quad \forall\, t\ge 0  \}) \le 
    	\Leb{1} \left([-1,1] \cap \bigcup_{r\in E} B_{2\bar c\eps}(r) \right) \le \frac{4 \bar c}{c} \eps.
    \end{equation}
    
    \medskip
    
    Let us now pass to the proof of (ii). Let $s\in [0, 1]$ be such that $d_{\TT^2}( t\xi, (t+s) w \xi')<\frac{\eps}{N}$ for some $t\ge 0$. Arguing as above we can assume that $t \xi =q \in A$, $ d_{\TT^2} (q , (t w +  s) \xi')\le 10\bar c\eps N^{-1}$
    and we can find $t'\in \RR$ and $k_1,k_2\in \mathbb{Z}$ such that
    $(t,t')\in T_{int}$ and $k_1 t_1 + k_2 t_2 =t$, $k_1 t_1' + k_2 t_2' = t'$.
    We have the following identity on $\TT^2$
    \begin{align*}
    (t w + s) \xi'  & = t \xi' + ( t ( w - 1 ) + s ) \xi'
    = (k_1 t_1+k_2 t_2) \xi' + ( t (w - 1) + s )  \xi'
    \\& =
    q + (k_1(t_1-t_1')+k_2(t_2-t_2') +  t (w-1) + s)  \xi'
    =q + (k_1(w t_1-t_1') + k_2(w t_2-t_2') + s) \xi'
    \end{align*}
    therefore, arguing as above we deduce $-s\in B_{10\bar c\eps N^{-1}}((k_1(wt_1-t_1') + k_2(wt_2-t_2'))) + \mathbb{Z}T'$.
    
    Notice now that the set $E:=\{  k_1( w t_1-t_1')+k_2( w t_2-t_2'): \, k_1, k_2\in \mathbb{Z} \}$ is discrete, any two consecutive points in $E$ have a fixed distance $c \ge c'(\xi, \xi')N^{-1}>0$ and $E + \mathbb{Z} T'= E$. In particular
    \begin{align}
    \Leb 1 ( [0,1] \setminus & \{s : \, d_{\TT^2}( t\xi, (t w + s) \xi')\ge  \eps \quad \forall\, t\ge 0  \})  \le 
    \Leb{1} \left([0,1] \cap \bigcup_{r\in E} B_{10\bar c\eps N^{-1}}(r) \right)
    \\ & \le \frac{2}{c'(\xi,\xi') N^{-1}} 10 \bar c \eps N^{-1}
    \le \frac{20 \bar c }{ c'(\xi,\xi')} \eps.\qedhere
    \end{align}   
\end{proof}

\begin{proof}
	[Proof of Proposition \ref{prop:disjointsupportdim2}]
Let us write
	\[
	\Lambda^k=\{\xi_{m,k}: \,m=1,...,m_0 \} \quad\text{for $k=1,2$}\, .
	\]
The key ingredient is Lemma \ref{lemma:disjointsupportdim2} and indeed the constant $c_0$ is chosen such that
\begin{equation}\label{e:choice_c_0}
2C  {\bar C} c_0 m_0 <1\, ,
\end{equation}
where $C$ is the constant appearing in Lemma \ref{lemma:disjointsupportdim2}(i)\&(ii). 

Notice that we are interested in pairs $(\xi_{m,k},n)$ such that $k\equiv n$ mod $2$. Without loss of generality we can thus assume that $k$ is a function of $n$ and takes the values $1$ or $2$ depending on the congruence class of $n$ modulo $2$. 
In particular we will use the shorthand notation $a_{m,n}$ for the point $a_{\xi_{m, k},n}$. We will find $a_{m,n}$ inductively, after endowing  
the set $\{1, \ldots , m_0\}\times \mathbb N\setminus \{0\}$ with the lexicographic order. More precisely we write $(m,n) \leq (m', n')$ if
\begin{itemize}
\item either $n<n'$ 
\item or $n=n'$ and $m<m'$.
\end{itemize} 
At the starting point of the induction we set $a_{1,1}=0$. For the inductive step, we fix $(m',n')$ and
	assume that $a_{m,n}$ has been already defined for any $(m,n) \leq (m',n')$. If $m'< m_0$ we need to define $a_{m'+1,n'}$, otherwise we have $m'=m_0$ and we need to define $a_{1, n'+1}$. We explain how to proceed just in the case $m'< m_0$, since the other case follows from the same argument. To fix ideas, let us assume that the congruence class of $n'$ is $1$, so that the congruence class of $n'-1$ is $2$. We look for 
	 $a_{m'+1, n'}\in [0,1]$ such that
	\begin{equation}\label{zz1}
		d_{\TT^2}(( t w_{n'} +a_{m'+1, n'} )\xi_{m'+1,1} , x^{n'}_{\xi_{m,1}}(t) ) \ge c_0 \ge \frac{c_0}{n'}
		\quad \text{for every $t>0$, when $1\le m < m'+1$}
	\end{equation}
	and
	\begin{equation}\label{zz2}
		d_{\TT^2} ( (t w_{n'}+a_{m'+1, n'} ) \xi_{m'+1,1} 
		, x^{n'- 1}_{\xi_{m, 2}}(t)) \ge \frac{c_0}{n'}
		\quad\text{for every $t>0$, when $1\le m \le m_0$}
	\end{equation}
(we interpret the latter condition as empty when $n'=1$). 
Define the sets $G^1_m,G^2_m \subseteq [0,1]$ as
\[
G^1_m:= \{ a: \; 		d_{\TT^2}( t w_{n'} \xi_{m'+1,1} + a \xi_{m'+1, n'}, x^{n'}_{\xi_{m,1}}(t) ) \ge c_0 \quad \text{for every $t>0$}\}
\]
\[
G^2_m:= \{ a: \; d_{\TT^2} ( t w_{n'} \xi_{m'+1,1} + a \xi_{m'+1, n'}, x^{n'- 1}_{\xi_{m, 2}}(t)) \ge \frac{c_0}{n'}
\quad\text{for every $t>0$} \}\, .
\]
Finally, let $G := G^1_1\cap \ldots \cap G^1_{m'}\cap G^2_1 \cap \ldots \cap G^2_{m_0}$ be their intersection.
Note that the existence of 
$
a_{m'+1, n'} 
$ is equivalent to the fact that $G$ is not empty. According to Lemma \ref{lemma:disjointsupportdim2}(i) $\Leb 1 ([0,1]\setminus G^1_m) \leq C  c_0$ for every $m\in \{1, \ldots , m'\}$, while according to Lemma \ref{lemma:disjointsupportdim2}(ii) $\Leb 1 ([0,1]\setminus G^2_m)\leq C { \bar C} c_0$ for every $m\in \{1, \ldots , m_0\}$. In particular, by our choice of $c_0$ in \eqref{e:choice_c_0} we have $\Leb 1 ([0,1]\setminus G) \leq C ( { \bar C} m_0 +m') c_0 \leq 2C \bar C m_0 c_0 <1$, which in turn implies that $G$ is not empty and completes the proof. 
\end{proof}

\subsection{Suitable discretized speeds} Clearly we cannt apply the geometric arrangement of the previous section if we choose $v_n =n$ for the values of the parameter $\sigma$ since the rations of $\mu_{q+1}^{d/p'} \left(\frac{n}{\kappa}\right)^{1/p'}$ are not even guaranteed to be rational. The aim of the next lemma is to show that it suffices to perturb $\{n\}_{n\in \mathbb N\setminus \{0\}}$ slightly to a new sequence $\{v_n\}_{n\in \mathbb N\setminus \{0\}}$ in order to achieve that the $\frac{w_{n+1}}{w_n}$ are rational numbers with a denominator which is not too large (in fact comparable to $n$). 

\begin{lemma}\label{lemma:velocitiesdim2}
	Fix $1\le p' < \infty$. Then there exist $\bar C>0$, $\bar N\in \NN$ depending only on $p'$, functions $A: \NN\setminus \{ 0,1 \}\to \{0,\bar N\}$, $N: \NN\setminus \{ 0,1 \}\to \NN\setminus \{ 0 \}$ and $v: \NN\setminus \{ 0 \}\to \RR^{+}$ such that, for every $n\ge 1$ the following holds:
	\begin{itemize}
		\item[(i)] $|v_n-n|\le 1$;
		\item[(ii)] $ \left( \frac{ v_n}{ v_{n-1} } \right)^{1/p'} = 1 + \frac{A_n}{N_n}$;
		\item[(iii)] $N_n< \bar C n$.
	\end{itemize}
\end{lemma}

\begin{proof}
We prove the statement by induction on $n$. For $n=1$ and $n=2$ we set $v_1=v_2=1$, $A_2=0$, $N_2= 1$, and the statement is satisfied. Suppose now that the claim is verified for $\bar n$. If $v_{\bar n}\geq \bar n$, we set $v_{\bar n+1}= v_{\bar n}$, $A_{\bar n+1} =0$, $N_{\bar n+1} = N_{\bar n}$ and the claim is verified. Hence we can assume that $v_{\bar n}< \bar n$.
We claim that we can choose $A_{\bar n+1} = \bar N$,   $v_{\bar n+1}$ and $ N_{\bar n+1}$ with 
\begin{equation}
\label{eqn:choice}
v_{\bar n+1} \in [\bar n+1,\bar n+2] \qquad \mbox{and} \qquad 
 N_{\bar n+1}= \bar N \Big[\Big(\frac{v_{\bar n+1}}{v_{\bar n}} \Big)^{1/p'} -1 \Big]^{-1}
\in \NN.
\end{equation}
Indeed, consider the continuous, decreasing function
$$f(t) =\bar N \Big[\Big(\frac{t}{v_{\bar n}} \Big)^{1/p'} -1 \Big]^{-1};$$
it is enough to show that
$$1 \leq f({\bar n+1}) - f(\bar n+2) = \bar N v_{\bar n}^{1/p'} \frac{(\bar n+2)^{1/p'}-_{\bar n+1}^{1/p'}}{(({\bar n+1})^{1/p'}- v_{\bar n}^{1/p'} ) ((\bar n+2)^{1/p'} - v_{\bar n}^{1/p'} )}=:\bar N g(\bar n, v_{\bar n})$$
 to find 
$t\in [\bar n +1,\bar n +2]$ such that $f(t) \in \NN$.
Since the function $g(\bar n, v_{\bar n})$ is increasing with respect to the variable $v_{\bar n}$,
$$ \bar N g(\bar n, v_{\bar n}) \geq \bar N g(\bar n, \bar n-1) = \bar N  \Big[ \Big(\Big(1+\frac{2}{\bar n-1} \Big)^{1/p'} -1 \Big)^{-1}- \Big(\Big(1+\frac{3}{\bar n-1} \Big)^{1/p'} -1 \Big)^{-1} \Big]
$$
We finally choose $\bar N:=\bar N (p')$ in such a way that $\inf_{\bar n \ge 2} g(\bar n, \bar n-1) \geq \bar N^{-1}$; we notice that this infimum is positive since the function $ g(\bar n, \bar n-1)$ is positive for every $\bar n\ge 2$ and, by a simple Taylor expansion, it grows linearly as $\bar n\to \infty$. This proves the claim \eqref{eqn:choice}.

With this choice of $v_{\bar n+1}$ and $N_{\bar n+1}$, 
 recalling also that $v_{\bar n}< \bar n$, we get that the statement (iii) is satisfied
$$
N_{\bar n+1}\leq   \bar N \Big[\Big(1+\frac{1}{\bar n} \Big)^{1/p'} -1 \Big]^{-1} \leq \bar C \bar n.\qedhere
$$
\end{proof}

\subsection{Disjointness of the supports}
Set $w_n:= \mu_{q+1}^{d/p'} \left(\frac{v_n}{\kappa}\right)^{1/p'}$ where $\{v_n\}_{n\ge 1}$ is given by Lemma \ref{lemma:velocitiesdim2}. We apply Proposition \ref{prop:disjointsupportdim2} to $\Lambda^1$, $\Lambda^2$ and $\{ w_n \}_{n\ge 1}$ (notice that the assumptions are satisfied in view of Lemma \ref{lemma:velocitiesdim2}) obtaining the family $\{ a_{\xi, n}:\, \xi\in \Lambda^{[n]} \}$. 
Finally, starting from the building blocks introduced in Section~\ref{sec:build} we define
\begin{equation}
	W_{\xi, \mu_{q+1}, v_n/\kappa}( t , x ):=\tilde W_{\xi, \mu_{q+1}, v_n/\kappa}( t, x-a_{\xi,n}\xi),
	\quad
	\Theta_{\xi, \mu_{q+1}, v_n/\kappa}(t, x):=\tilde \Theta_{\xi, \mu_{q+1}, v_n/\kappa}(t, x-a_{\xi,n} \xi),
\end{equation}
for any $n\ge 1$ and $\xi\in \Lambda^{[n]}$, as already explained.

We therefore now show \eqref{eq:disjointsupportsdim2}, namely that
\begin{equation}\label{eq:disjointsupportsdim2-bis}
	W_{\xi, \mu_{q+1}, v_n/\kappa} \cdot \Theta_{\xi', \mu_{q+1}, v_m/\kappa}( t , x )=0
	\quad \text{for any $(x,t)\in \TT^2\times \RR^+$,}
\end{equation}
whenever $\xi\neq \xi'$, $|n-m|\le 1$ and $n,m\le  C  \lambda_{q}^{d(1+\alpha)+1}$.

Indeed for any fixed $t\ge 0$ one has the inclusions 
\begin{equation}
     \supp W_{\xi, \mu_{q+1}, v_n/\kappa}( t , \cdot) \subset B_{2\rho\mu^{-1}_{q+1}}( t w_n \xi + a_{\xi, n}\xi ),
     \quad
     \supp \Theta_{\xi', \mu_{q+1}, v_m/\kappa}( t , \cdot) \subset B_{\rho\mu^{-1}_{q+1}}( t w_m \xi' + a_{\xi', m} \xi'),
\end{equation}
hence we just need to check that
$B_{2\rho\mu^{-1}_{q+1}}( t w_n \xi + a_{\xi, n}\xi )\cap B_{\rho\mu^{-1}_{q+1}}( t w_m \xi' + a_{\xi', m} \xi')= \emptyset$.
Proposition \ref{prop:disjointsupportdim2} guarantees 
\[
d_{\TT^2}(t w_n \xi + a_{\xi, n}\xi,  t w_m \xi' + a_{\xi', m} \xi')\ge \frac{c_0}{n},
\]
hence the claim is proven provided 
\begin{equation}\label{zz3}
	\frac 34 \mu_{q+1}^{-1} \le \frac{c_0}{C \lambda_{q}^{d(1+\alpha)+1}}
	.
\end{equation}
The proof of \eqref{zz3} follows from the choice of $\mu_{q+1} = \lambda_q^{b\gamma}$, and since $\gamma>1$, $b>d(1+\alpha)+1$.

\subsection{{Proof of the Proposition~\ref{prop:inductive} in the case d=2}}
The estimates up to Section~\ref{sec:end} are done in the same exact way, up to observing that $v_n/\kappa$ is comparable to $n/\kappa $ up to a factor $2$.
In Section~\ref{sec:newerror}, we compute in \eqref{eqn:formulone} the product of $\theta_{q+1}^{(p)}$ and $ w_{q+1}^{(p)} $ in order to see the cancellation of the old error $R_\ell$;
now it has the expression
\begin{equation}
	\theta_{q+1}^{(p)} w_{q+1}^{(p)} = \sum_{n\ge 12} \sum_{\xi\in \Lambda^{[n]}} \chi(\kappa|R_{\ell}|-n) a_{\xi}\left(\frac{R_\ell}{|R_\ell|}\right) \Theta_{\xi, \mu_{q+1}, v_n /\kappa} W_{\xi, \mu_{q+1}, v_n/\kappa}(\lambda_{q+1}t, \lambda_{q+1} x)
\end{equation}
as a consequence of \eqref{eq:disjointsupportsdim2}, \eqref{remark:sum is finite}, the fact that $\chi\cdot \bar \chi= \chi$ and $\chi ( \kappa | R_{\ell} | - n ) \cdot \chi( \kappa| R_{\ell}| - n )=0$ when $|n-m|>1$.

Since the average of $\Theta_{\xi, \mu_{q+1}, n /\kappa}W_{\xi, \mu_{q+1}, n /\kappa}$ which appears from the forth line of formula \eqref{eqn:formulone}, in the definition of $R^{quadr}$ and in $m$ is now ${v_n}/{\kappa}\xi$ rather than $n/\kappa\xi$, the definition of $\tilde R_\ell$ should now be replaced by
$$
\tilde R_\ell:= 
\sum_{n\ge \nop} \chi(\kappa |R_\ell|- n)	\frac{R_\ell}{|R_\ell|} \frac{v_n}{\kappa},
$$
and the obvious modification takes place for the definition of $R^{quadr}$ and $m$.
The estimate \eqref{eqn:tildeRestimate} now works analogously to give $|R_\ell- \tilde R_\ell| \leq \frac{17
}{20
} \delta_{q+2}$.
The rest of the estimates work as in Sections~\ref{sec:newerror} and~\ref{sec:highderiv}.

 \section{Proof of the ill-posedness theorems}
\subsection{Proof of Theorem~\ref{thm:main-CE}}

Without loss of generality we assume $T=1$.
Let $\alpha,b, a_0, M>5$, $\beta>0$ be fixed as in Proposition~\ref{prop:inductive}. Let $a \geq a_0$ be chosen such that
$$ \sum_{q=0}^\infty \delta_{q+1}^{1/p} \leq \frac 1 {32M}.$$
Let $\chi_0$ be a smooth time cutoff which equals $1$ in $[0,1/3]$ and $0$ in $[2/3,1]$, 

We set $ \lambda = 20a$ and define the starting triple $(\rho_0,  u_0, R_0)$ of the iteration as follows:
\[
\rho_0 = \chi_0(t) + \Big(1+\frac{\sin( \lambda x_1)}{4}\Big) (1- \chi_0(t)), \qquad u_0 = 0, \qquad R_0 =- \partial_t \chi_0\frac{ \cos (\lambda x_1)}{4 \lambda} e_1\, .
\]
Simple computations show that the tripe enjoys \eqref{eqn:CE-R} with $q=0$. Moreover $\|R_0\|_{L^1} \leq C \lambda^{-1} = C \lambda_0^{-1}$ and thus \eqref{eqn:ie-1} is satisfied because $2\beta <1$ (again we need to assume $a_0$ sufficiently large to absorb the constant). Next $\|\partial_t \rho_0\|_{C^0}+\|\rho_0\|_{C^1}\leq C \lambda \leq C \lambda_0$. Since $u_0\equiv 0$ and $\alpha >1$, we conclude that \eqref{eqn:ie-2} is satisfied as well. 

Next use Proposition~\ref{prop:inductive} to build inductively $(\rho_q, u_q, R_q)$ for every $q\geq 1$. The sequence $\{\rho_q\}_{q\in \NN}$ is Cauchy in $C (L^{p})$ and we denote by $\rho \in  C ([0,1], L^{p})$ its limit.
Similarly the sequence of divergence-free vector fields $\{u_q\}_{q\in \NN}$ is Cauchy in $C([0,1], L^{p'})$ and $C([0,1], W^{1,r})$; hence, we define $u \in  C^0 ([0,1], L^{p'}\cap W^{1,r})$ as its (divergence-free) limit.

Clearly $\rho$ and $u$ solve the continuity equation and $\rho$ is nonnegative on $\TT$ by
\[
\inf_{\TT} \rho \geq \inf \rho_0 + \sum_{q=0}^\infty  \inf (\rho_{q+1} - \rho_q) \geq \frac 34 - \sum_{q=0}^\infty \delta_{q+1}^{1/p} \geq \frac 14\, .
\]
Moreover, $\rho$ does not coincide with the solution which is constantly $1$, because
\[
\|\rho - 1\|_{L^p} \geq \|1-\rho_0\|_{L^p} - \sum_{q=0}^\infty  \|\rho_{q+1}-\rho_q\|_{L^p}  \geq  \frac 1 {16}- M \sum_{q=0}^\infty \delta_{q+1} >0.
\]
Finally, since $\rho_0(t, \cdot) \equiv 1 $ for $t\in [0,1/3]$, point (c) in Proposition~\ref{prop:inductive} ensures that  $\rho (t, \cdot) \equiv 1$ for every $t$ sufficiently close to $0$.

\subsection{Proof of Theorem~\ref{thm:main-flow}}
We first recall a general fact: if $u$ is an everywhere defined Borel vector field in $L^1([0,T] \times \TT; \RR^d)$ such that, for a.e. $x \in \TT$, the integral curve starting from $x$ is unique, then the corresponding continuity equation is well posed in the class of nonnegative, $L^1([0,T] \times \TT)$ solutions for any $L^1$ initial datum. 

Indeed, Ambrosio's superposition principle (see e.g. \cite[Theorem 3.2]{Luigi-CIME}) guarantees that each nonnegative, $L^1([0,T] \times \TT)$ solution is transported by integral curves of the vector field, namely (no matter how the Borel representative is chosen) there is a probability measures $\eta$ on the space of absolutely continuous curves, supported on the integral curves of the vector field in the sense of Definition~\ref{defn:int-curve},  such that $\rho(t,x)\, \mathscr{L}^d = (e_t)_\# \eta$ for a.e. $t\in [0,T]$. Let us consider the disintegration $\{\eta_x\}_{x \in \TT^d}$ of $\eta$ with respect to the map $e_0$, which is $\rho_0$-a.e. well defined; since by assumption for a.e. $x \in \TT$, the integral curve starting from $x$ is unique (and hence coincides with the regular Lagrangian flow), we deduce that $\eta_x$ is a Dirac delta on the curve $t \to X(t,x)$ and consequently $\rho(t,\cdot ) \mathscr{L}^d = X(t,\cdot)_\# (\rho_0 \mathscr{L}^d)$. This concludes the proof of the first claim.

Let $u$ be the vector field given by Theorem~\ref{thm:main-CE} and observe that the Cauchy problem for the continuity equation \eqref{eqn:CE} from the initial datum $\rho_0 \equiv 1$ has two different nonnegative solutions in $[0,T]$: $\rho^{(1)} \equiv 1$ and the nonconstant solution $\rho^{(2)}$ given by Theorem~\ref{thm:main-CE}. Hence, by the previous observation we conclude that there exists a set of initial data of positive measure such that the corresponding integral curves are nonunique. Since the fact that the two functions are distinct solutions of the continuity equation is independent of the pointwise representative chosen for the vector field, this completes the proof of Theorem \ref{thm:main-flow}.

\section{Asymmetric Lusin-Lipschitz estimates}

\subsection{Proof of Proposition \ref{prop:LusinLipschitzasimmetrica}}
We will prove the inequality up to constants and we assume $\alpha\ge 1/d$, since for any $\alpha\in (0,1/d)$ a simple application of the Young inequality gives
\begin{equation*}
	g(x)+g(x)^{\frac{1}{d}}g(y)^{1-\frac{1}{d}}
	=g(x)+g(x)^{\frac{1-\alpha d}{(1-\alpha)d}}\left(g(x)^{\frac{\alpha(d-1)}{(1-\alpha)d}} g(y)^{\frac{d-1}{d}} \right)
	\le C({\alpha,d})( g(x)+g(x)^{\alpha}g(y)^{1-\alpha}).
\end{equation*}
We next introduce the following localized Hardy Littlewood maximal function: regarding any integrable function $f:\TT\to\RR$ as a periodic function on $\mathbb R^3$, we set
\[
M f (x):= \sup_{0\leq R \leq 3} \frac{1}{R^d}\int |f| (z)\, dz\, .
\]
We will show below that the conclusion of the proposition hold for
\[
g (x) := (M|Du|^q)^{1/q}\, 
\]
where $q=\alpha d$. In particular the map $L^r \ni Du \mapsto g\in L^r$ is continuous. 

Note first that it suffices to prove the estimate for $x,y\in \{g<\infty\} \subset \{M|Du|<\infty\}$. On $\{g = \infty\}$ we can arbitrarily define $u$ to be $0$: this will not matter for our purposes because when one of the two points $x,y$ belong to $\{g = \infty\}$ the right hand side of \eqref{eq:LusinLipparzialmenteAsimmetrica} is infinite, making the inequality trivial. On $\{g<\infty\}$ we wish instead to define $u$ everywhere in a sensible way. We fix thus a smooth convolution kernel $\varphi$ supported in the ball of radius 1, assume $x\in \{g<\infty\}$ and consider $u_k:= u * \varphi_{2^{-k}}$. Recalling the Poincar\'e inequality
\[
\frac{1}{2^{kd}}\int_{B_{2^{-k}} (x)} |u(z) - u_k (x)|\, dz \leq C 2^{-k} M|Du (x)| \leq C 2^{-k} g(x)\, 
\]
(where the constant depends on $\varphi$) we infer $|u_{k+1} (x) - u_k (x)|\leq C 2^{-k} g (x)$.
This implies that $\{u_k (x)\}_k$ is a Cauchy sequence and has a limit: we define then $u(x)$ to be such limit.

We next fix $x,y\in \{g<\infty\}$, regard $u$ as a periodic function defined on the whole $\RR^d$ and set $R:= |x-y|$. W.l.o.g. $R\leq 1$. Moreover we recall the classical inequality
	\begin{equation}\label{eq:well-known}
	|u(x)-u(y)|\le C(d)
	\left( \int_{B_{R}(x)} \frac{|D u(z)|}{|x-z|^{d-1} }\, dz+ \int_{B_{R}(y)} \frac{|D u(z)|}{|x-z|^{d-1} }\, dz  \right)\, .
	\end{equation}
When $u\in C^1$ we refer the reader to \cite[Lemma 3.1]{J10} for a proof. Otherwise, the inequality can be validated passing to the limit on the respective ones for the approximating functions $u_k$'s (using that $\lim_k u_k (x) = u(x)$, $\lim_k u_k (y) = u(y)$ and standard facts about convolutions). 
A classical telescoping argument gives next
	\begin{equation}\label{eq:telescopingargument}
	\int_{B_R(x)} \frac{|D u(z)|}{|x-z|^{d-1} }\, dz
	=\sum_{k=0}^{+\infty} \int_{B_{2^{-k}R}(x)\setminus B_{2^{-k-1}R}(x)} \frac{|D u(z)|}{|x-z|^{d-1}}\, dz
	\le C(d) R M |D u|(x)\, .
	\end{equation}
	Recall that $\alpha \in (1/d,r/d)$ and fix $\eps\in(0,1]$ to be chosen later.
	We write
	\begin{equation}\label{eq:split}
	\int_{B_{R}(y)} \frac{|D u(z)|}{|y-z|^{d-1} }\, dz 
	=\int_{B_{R}(y)\setminus B_{\eps R}(y)} \frac{|D u(z)|}{|y-z|^{d-1} }\, dz 
	+\int_{B_{\eps R}(y)} \frac{|D u(z)|}{|y-z|^{d-1} }\, dz
	:= I+II.
	\end{equation}
	Let us study $I$ and $II$ separately. Using the H\"older inequality we get
	\begin{align*}
	I & \le \left( \int_{B_{R}(y)} |D u(z)|^q \, d z\right)^{1/q} \left( C(d)\int_{\eps R}^{R} \frac{1}{s^{(d-1)q'-d+1}}  \, ds\right)^{1/q'}\\
	&\le C({d,q}) R^{d/q}
	\left( \frac{1}{R^d} 
	\int_{B_{3R}(x)} |D u(z)|^q \, d z
	\right)^{1/q} \frac{1}{(R\eps)^{d-1-d/q'}}\\
	&\le C({d,q}) R g(x)\eps^{d/q'-d+1}=C({d,q}) R g(x)\eps^{-d/q+1},
	\end{align*}
	where $\frac{1}{q'}=1-\frac{1}{q}$.
	For what concerns $II$, we argue as in \eqref{eq:telescopingargument} getting
	\begin{equation*}
	II \le C(d) R \eps M |D u| (y)\le C({d})R \eps g(y).
	\end{equation*}
	Putting the two estimates together and choosing 
	\begin{equation*}
	\eps =
	\begin{cases}
	\left(\frac{g(x)}{g(y)}\right)^{q/d} & \text{for}\ g(y)\ge g(x)\\
	1& \text{otherwise}
	\end{cases}
	\end{equation*} 
	we obtain
	\begin{equation*}
	\int_{B_{R}(y)} \frac{|D u(z)|}{|y-z|^{d-1} } \, dz	
	\le C({d,\alpha}) R (g(x)+ g(x)^{\alpha}g(y)^{1-\alpha}),	
	\end{equation*}
	that along with \eqref{eq:telescopingargument} gives the sought conclusion.

\subsection{A second version of the asymmetric Lusin-Lipschiz} A simple application of the Young inequality gives the following linear version of \eqref{eq:LusinLipparzialmenteAsimmetrica}.
\begin{corol}\label{cor:LusinLipsLinearversion}
	Let $u\in W^{1,r}(\TT)$ for some $1<r\le d$. Then,
	for any $q\in [r, r\frac{d-1}{d-r})$ there exist a positive constant $C:= C({r,d,q})$ and nonnegative functions $a\in L^1$ and $b\in L^q$ satisfying
	\begin{equation}\label{eq:z}
	\|a\|_{L^1} \le C\|D u\|_{L^r}
	\quad \text{and}\quad
	\|b\|_{L^q}\le C \|D u\|_{L^r},
	\end{equation}
	together with
	\begin{equation*}
	|u(x)-u(y)|\le |x-y|(a(x)+b(y))
	\quad\text{for any $x,y\in \RR^d\setminus N$}.
	\end{equation*}
	Moreover, we can take $N=\emptyset$ provided we choose a suitable representative of $u$, $a$ and $b$ and the latter can be selected so that the respective map $W^{1,r} \ni u \to (a,b)\in L^1\times L^q$ is continuous. 
\end{corol}

We will now show that the range of exponents above is optimal. First of all we prove the following simple proposition.

\begin{propos}\label{prop:exampleLusin}
	Let $d\ge 2$, $0<\beta<d$ and $q>1$. If there exist $g\in L^1(B_1)$ and $h\in L^q(B_1)$ such that
	\begin{equation}\label{z}
	\left|\frac{1}{|x|^{\beta}}-\frac{1}{|y|^{\beta}}\right|
	\le |x-y| (g(x)+h(y))
	\quad\forall x,y\in B_1\subset \RR^d,
	\end{equation}
	then $q\le \frac{d-1}{\beta}$.
\end{propos}
\begin{proof}
	Fix $\alpha>0$. Plugging $y=|x|^{\alpha}x$ in \eqref{z}, dividing by $|x|(1-|x|^{\alpha})$ and integrating in $B_{1/2}(0)$, we get
	\begin{equation}
	\int_{B_{1/2}(0)} \frac{1-|x|^{\beta \alpha}}{1-|x|^{\alpha}}\frac{1}{|x|^{\beta(\alpha+1)+1}} \, d x
	\le \int_{B_{1/2}(0)}g(x) \, d x+\int_{B_{1/2}(0)}h(x|x|^{\alpha})\, d x.
	\end{equation}
	By changing variables in the last integral, according to $y=x|x|^{\alpha}$, we end up with
	\begin{equation}\label{z1}
	\int_{B_{1/2}(0)}\frac{1}{|x|^{\beta(\alpha+1)+1}} \, d x\le
	C\int_{B_{1/2}(0)}g(x) \, d x
	+C\int_{B_{1}(0)}h(y)\frac{1}{|y|^{d\frac{\alpha}{\alpha+1}}} \, d y.
	\end{equation}
	By the H\"older inequality the last integral in \eqref{z1} is finite for any $\alpha<q-1$;
	therefore, \eqref{z1} implies $\beta(\alpha+1)<d-1$ for any  $\alpha <q-1$. This easily gives $q\le \frac{d-1}{\beta}$.
\end{proof}

Let us now fix $1\le r<d$. For any $\beta<d/r-1$ consider the function $u(x):=|x|^{-\beta}\in W^{1,r}_{\text{loc}}(\RR^d)$ and cut it off with a smooth cut-off function so that it is compactly supported in $(-1/2,1/2)^d$. Extend then the function by periodicity and regard it as a function in $W^{1,r} (\TT)$. Proposition \ref{prop:exampleLusin} ensures that the exponent $q$ in Corollary \ref{cor:LusinLipsLinearversion}, associated to $u$, must satisfy $q\le \frac{d-1}{\beta}$ and therefore $q\le r\frac{d-1}{d-r}+\eps$ with $\eps\to 0$ when $\beta\to d/r-1$.

\subsection{The critical case $p=d$}
We discuss here possible improvements of \eqref{eq:LusinLipparzialmenteAsimmetrica} in the critical case $p=d$. First of all observe that, in general, we cannot expect \eqref{e:tot_asimmetrica} to hold for $u\in W^{1,d}(\TT)$ since it would in turn imply $u\in L^{\infty}$. However, following closely the proof of Proposition \ref{prop:LusinLipschitzasimmetrica} one can show:
\begin{equation}\label{eq:criticalcase1}
|u(x)-u(y)|\le C(d)|x-y|\left( 1+ M |D u|(x)\left(1+\log\left(  M |D u|^d (y)  \right)^{\frac{d-1}{d}} \right)\right)
\quad \text{for any $x,y\in\RR^d\setminus N$},
\end{equation}
where $N\subset \RR^d$ is negligible. We do not give the details since \eqref{eq:criticalcase1} does not play any role in the sequel.
We instead show a generalization of \eqref{eq:LusinLipparzialmenteAsimmetrica}
for maps with $D u$ in the Lorentz space $L^{d,1}$ (see e.g. \cite{Grafakos} for the relevant definition): as a corollary we get the $\Leb d$-a.e. uniqueness of trajectories of vector fields enjoying such regularity. We recall, in passing, that the assumptions $Du\in L^{d,1}$ implies the continuity of $u$. 

\begin{propos}\label{prop:Lorentz}
	Assume $u\in W^{1,1} (\TT)$ satisfy $D u\in L^{d,1}$. Then there exists $g\in L^{d,\infty}$ such that
	\begin{equation}\label{eq:Ld1}
	\| g\|_{L^{d,\infty}}\le C(d) \|D u\|_{L^{d,1}},
	\end{equation}
	\begin{equation*}
	|u(x)-u(y)|\le (C(d) M |D u| (x) + g (x)) |x-y|
	\quad \text{for any }x,y\in \RR^d\setminus N
	\end{equation*}
for some negligible set $N$. The latter can be assumed to be empty if $u$ is appropriately defined pointwise and moreover there is a continuous selection map $Du \ni L^{d,1} \mapsto g\in L^{d,\infty}$. 
\end{propos}

\begin{proof}
	Fix $x,y\in \RR^d$ and $R=2|x-y|$ and argue as in the proof of Proposition \ref{prop:LusinLipschitzasimmetrica}. Our conclusion will follow from \eqref{eq:well-known} and \eqref{eq:telescopingargument} provided we show
	\begin{equation}
	\int_{B_R(y)} \frac{|D u(z)|}{|y-z|^{d-1}} \, d z
	\le R\, g(x)
	\end{equation}
	for some $g\in L^{d,1}(\RR^d)$ satisfying \eqref{eq:Ld1}.
	
	The H\"older inequality for Lorentz spaces gives
	\begin{equation*}
	\int_{B_R(y)} \frac{Du (z)}{|y-z|^{d-1}} \, d z
	\le C(d) \||D u|{\mathbf{1}}_{B_R(y)}\|_{L^{d,1}}\||x-\cdot|^{1-d}\|_{L^{\frac{d}{d-1},\infty}}
	=C(d) \||D u|{\mathbf{1}}_{B_R(y)}\|_{L^{d,1}}.
	\end{equation*} 
	Observe that 
	\begin{equation*}
	\||D u|{\mathbf{1}}_{B_R(y)}\|_{L^{d,1}}\le \||D u|{\mathbf{1}}_{B_{3R}(x)}\|_{L^{d,1}}
	\le 3R \sup_{0<t<3}t^{-1}\||D u|{\mathbf{1}}_{B_t(x)}\|_{L^{d,1}},
	\end{equation*}
	let us set $g(x):=\sup_{0<t<3}t^{-1}\||D u|{\mathbf{1}}_{B_t(x)}\|_{L^{d,1}}$ and check \eqref{eq:Ld1}.
	First notice that
	\begin{align*}
	g(x)=&\sup_{0<t<3}t^{-1}\||D u|{\mathbf{1}}_{B_t(x)}\|_{L^{d,1}}
	=\sup_{0<t<3}\int_0^{\infty} \left(\frac{1}{t^d}\Leb d (\left\lbrace|D u|>\lambda\right\rbrace\cap B_t(x)) \right)^{1/d} \, d \lambda\\
	\le & C(d)\int_0^{\infty} \left[ M({\mathbf{1}}_{|D u|>\lambda})(x) \right]^{1/d} \, d \lambda,
	\end{align*}
    where in the latter estimates we regard $|Du|$ as a function on the torus.

	Now we argue by duality. Fix $h\in L^{\frac{d}{d-1},1}$. Recall that $\|g^d\|_{L^{1,\infty}}=\|g\|_{L^{d,\infty}}^d$ for any nonnegative $g\in L^{d,\infty}$. Hence, using the weak $(1,1)$ estimate for the
	maximal function, we get
	\begin{align*}
	\int g(x) h(x) \, d x \le & C(d) \int_0^{\infty}\int \left[ M({\mathbf{1}}_{|D u|>\lambda})(x) \right]^{1/d}h(x)\, dx \,d \lambda \\
	\le& C(d)\int_0^{\infty}\|\left[ M({\mathbf{1}}_{|D u|>\lambda}) \right]^{1/d}\|_{L^{d,\infty}}\|h\|_{L^{\frac{d}{d-1},1}} \, d \lambda\\
	=& C(d) \int_0^{\infty}\|M({\mathbf{1}}_{|D u|>\lambda})\|_{L^{1,\infty}}^{1/d} \, d \lambda \ \| h\|_{L^{\frac{d}{d-1},1}}\\
	\le &C(d)\int_0^{\infty} \Leb d(\left\lbrace|D u|>\lambda\right\rbrace)^{1/d} \, d \lambda\ \|h\|_{L^{\frac{d}{d-1},1}}\\
	=& C(d) \|D u\|_{L^{d,1}} \|h\|_{L^{\frac{d}{d-1},1}}
	\end{align*}
	Since $h\in L^{\frac{d}{d-1},1}$ is arbitrary, by duality we get the desired estimate (see e.g. \cite[Theorem 1.4.17]{Grafakos}).
\end{proof}

\section{Well posedness theorems}

First of all we observe that, arguing as in \cite[Corollary 5.4]{CC18}, Proposition \ref{prop:Lorentz} implies the following result.

\begin{corol}\label{c:Lorentz}
	Let $u\in L^1([0,T], W^{1,1} (\TT))$ satisfy $|D u|\in L^1([0,T], L^{d,1}(\TT))$ and $\div u\in L^1([0,T],  L^{\infty}(\TT))$. For $\Leb d$-a.e. $x\in \TT$ there exists a unique trajectory of $u$ starting at $x$ at time $t=0$.
\end{corol}

\begin{proof}
Let $X$ be the regular Lagrangian flow associated to $u$, which exists by the DiPerna-Lions theory. We wish to show that for a.e. $x$ the curve $t\to X(t,x)$ is in fact the unique trajectory of the ODE. Consider the function $h(t,x):= M|Du| (t,x) + g (t,x)$, where $g (t, \cdot)$ is the map given by Proposition \ref{prop:Lorentz} when applied to $u(t, \cdot)$ (we choose $u\mapsto g$ continuously in order to avoid measurability issues). Observe that, by the usual change of coordinates formula,
\[
\int_\TT \int_0^T h (t, X(t,x)) \, dt\, dx \leq C \int_0^T \int_\TT h (t,x)\, dx\, dt < \infty\, . 
\]
In particular for a.e. $x$ we have that that $t\mapsto \gamma (t) := X (t,x)$ is an absolutely continuous trajectory solving $\dot{\gamma} (t) = u (t, \gamma (t))$ and that $ a(t):= h(t,\gamma (t)) = M |Du| (t, \gamma (t))+g(t, \gamma (t))\in L^1((0,T))$. 
Fix such an $x$ and assume $\bar\gamma$ is another absolutely continuous trajectory solving $\dot{\bar\gamma} (t) = u (t, \bar\gamma (t))$ and $\bar \gamma (0)=x$. It then follows that $f (t):= |\gamma (t) - \bar{\gamma} (t)|$ is absolutely continuous and that
\[
f' (t) \leq |u (t, \gamma (t)) - u(t, \bar\gamma (t))|\leq C {( M |Du| (t, \gamma (t))+g(t, \gamma (t)))}
 f(t) = C a(t) f(t)\, .
\]
Since $f (0) = 0$ and $a\in L^1$, it follows from Gronwall's Lemma that $f\equiv 0$ on $[0,T]$. 
\end{proof}


\begin{proof}[Proof of Theorem \ref{thm:wellposedness}]
By Ambrosio's superposition principle (see \cite[Theorem 3.2]{Luigi-CIME}) there exists a family of probability measures $\left\lbrace \eta_x\right\rbrace_{x\in \RR^d}\subset \Pr(AC([0,T],\TT))$ concentrated on integral curves of $u$, starting from $x\in\TT$ at time $t=0$, such that
	\begin{equation}\label{eq:superposition}
	\int_{\RR^d} \phi(x) \rho(t,x) d x=\int_{\RR^d} \left( \int \phi(\gamma(t))\, d\eta_x(\gamma)\right) \rho_0(x) \, dx
	\quad\text{for any $\phi\in C_c(\RR^d)$}.
	\end{equation}
	Let us also recall that under our assumptions on $u$ there exists a unique regular Lagrangian flow $X$ associated to it (see \cite{DL}). 
	The sought conclusion follows from the following claim: for $\rho_0\Leb d$-a.e. $x\in\RR^d$, $\eta_x$ is concentrated on the curve $t\to X(t,x)$. 
	
	We prove the claim just in the case $1<r\le d$. The case $r>d$ follows from the fact that we have classical uniqueness of the trajectories for a.e. initial data, as observed by \cite[Corollary 5.2]{CC18} (we can of course use Corollary \ref{c:Lorentz} as well, since $L^p\subset L^{d,1}$ for every $p>d$). For any $t\in [0,T]$ we consider a representative of $u(t,\cdot )\in W^{1,r}(\TT, \RR^d)$ such that Corollary \ref{cor:LusinLipsLinearversion} holds with $N=\emptyset$ for some $a_t\in L^1_{\text{loc}}$ and $b_t\in L^{p'}$ satisfying \eqref{eq:z} (note that \eqref{eq:wellposednessrange} guarantees $p'\in (r, r\frac{d-1}{d-r})$). Note that by the last statement of Corollary \ref{cor:LusinLipsLinearversion} we can ignore any measurability issue in the variable $t$. 
	
	For any $\gamma\in AC([0,T],\RR^d)$ integral curve of $u$, and any $x\in \RR^d$ one has
	\begin{equation*}
	\frac{d}{d t} |X(t,x)-\gamma(t)|\le |u(t,X(t,x))-u(t,\gamma(t))|\le |X(t,x)-\gamma(t)|(a_t(X(t,x))+b_t(\gamma(t)))
	\end{equation*}
	for a.e. $t\in [0,T]$. Therefore the Gronwall lemma guarantees $X(\cdot,x)=\gamma$, provided
	\begin{equation}
	\gamma(0)=x
	\quad \text{and}\quad 
	\int_0^T \left(a_t(X(t,x))+b_t(\gamma(t))\right) \, dt<\infty.
	\end{equation}
	Therefore our claim follows from
	\begin{equation}\label{z2}
	\int_0^T a_t(X(t,x)) \, dt<\infty
	\quad \text{and}\quad 
	\int_0^t b_t(\gamma(t)) \, dt<\infty
	\quad \text{for $\eta_x$-a.e. $\gamma$}
	\end{equation}
	for $\rho_0\Leb d$-a.e. $x\in \TT$.
	
	The first one is a consequence of
	\begin{equation*}
	\int_\TT \left( \int_0^T a_t(X(t,x)) d t\right) \, dx
	\le C \int_0^T \int_{\TT} a_t(x) \, dx \, dt
	\le C \int_0^T \| D u_t\|_{L^r} \, dt<\infty
	\end{equation*}
	where the constant $C>0$ depends on the compressibility constant in Definition \ref{def:regflow}. Here we have used \eqref{eq:z}.
	The second inequality in \eqref{z2} follows from
	\begin{align*}
	\int_{\TT}\left( \int \int_0^T b_t(\gamma(t)) \, dt d\eta_x(\gamma)\right) \rho_0(x) \, dx
	=&\int_0^T \int_{\TT} b_t(x)\rho(t,x) \, dx \, dt
	\le \left( \int_0^T \|b_t\|_{L^{p'}}\, dt\right) \|\rho\|_{L^{\infty}(L^p)} \\
	\le & C \| D u\|_{L^1(L^r)}\|\rho\|_{L^{\infty}(L^p)}<\infty
	\end{align*}
	where we have used \eqref{eq:superposition} and \eqref{eq:z}.
\end{proof}

\textbf{ Acknowledgements}. EB wishes to thank Elio Marconi and Paolo Bonicatto for many helpful discussions.
MC has been supported by
the SNSF Grant 182565 and by the NSF under Grant No. DMS-1638352. MC acknowledges gratefully the hospitality of the Institute for Advanced Studies, where part of this work was done. 
CDL has been supported by the NSF under Grant No. DMS-1946175.

\bibliographystyle{plain}
\bibliography{NSE}
\end{document}